\documentclass[12pt]{amsart}
\usepackage{amsmath,amsthm,amsfonts,latexsym,amssymb,amscd,color}
\usepackage{array}
\usepackage{xspace}
\usepackage{chemarr}
\usepackage{rotating,float}
\newtheorem{thm}{Theorem}[section]
\newtheorem{cor}[thm]{Corollary}
\newtheorem{lm}[thm]{Lemma}

\newtheorem{clm}[thm]{Claim}
\newtheorem*{clm*}{Claim}
\theoremstyle{definition}
\newtheorem{df}[thm]{Definition}

\newtheorem{exmp}[thm]{Example}
\newtheorem{prb}[thm]{Problem}

\newtheorem{glassumption}[thm]{Global Assumption}
\numberwithin{equation}{section}

\newcommand{\al}[1]{\mathbf{#1}}

\DeclareMathOperator{\id}{\textrm{id}}     

\newcommand{\var}[1]{{\mathcal{\uppercase{#1}}}}

\newcommand{\restr}{{\restriction}}

\newcommand{\LL}{\ensuremath{\mathcal{L}}\xspace}
\newcommand{\MM}{\ensuremath{\mathcal{M}}\xspace}
\newcommand{\RR}{\ensuremath{\mathcal{R}}\xspace}
\newcommand{\CC}{\ensuremath{\mathcal{C}}\xspace}
\newcommand{\SD}{\ensuremath{\mathcal{SD}}\xspace}
\newcommand{\RRn}{\RR_{[n]}}
\newcommand{\LT}{\mathsf{LT}}
\newcommand{\EE}{\mathsf{E}}

\newcommand{\ModA}{\text{\rm Mod}_A}
\newcommand{\arity}{\text{\rm arity}}
\newcommand{\cross}[1]{\mathsf{X}_{#1}}
\newcommand{\Prn}{{\textstyle\Pr_{[n]}}}
\newcommand{\PrA}{{\textstyle\Pr_A}}
\newcommand{\PrB}{{\textstyle\Pr_B}}

\let\phi=\varphi
\let\epsilon=\varepsilon
\let\bar=\overline

\DeclareMathOperator{\Mod}{Mod}

\newcommand{\qedsymbdiamond}{\renewcommand{\qedsymbol}{$\diamond$}}
\newcommand{\qeddiamond}{\qedsymbdiamond\qed}

\begin{document}

\title[Random Models]{Random Models of\\
  Idempotent Linear Maltsev Conditions.\\
  I. Idemprimality}

\author{Clifford Bergman}
\address[Clifford Bergman]{Department of Mathematics\\
Iowa State University\\
Ames, Iowa 50011\\
USA}
\email{cbergman@iastate.edu}

\author{\'Agnes Szendrei}
\address[\'Agnes Szendrei]{Department of Mathematics\\
University of Colorado\\
Boulder, CO 80309-0395\\
USA}
\email{Szendrei@Colorado.EDU}

\thanks{This material is based upon work supported by the
  National Science Foundation grants
  no.~DMS~1500218 and DMS~1500254. The second author
  acknowledges the support of 
  the Hungarian National Foundation for Scientific Research (OTKA)
  grant no.~K115518.} 
\keywords{Maltsev condition, idemprimal, Murski\v \i}
\subjclass[2010]{08B05, 08A40}
\date{\today}

\begin{abstract}
We extend a well-known theorem of Murski\v{\i} to the probability space of
finite models of a system $\MM$ of identities of a strong idempotent
linear Maltsev condition. We characterize the models of $\MM$ in a way that
can be easily turned into an algorithm for producing random finite models of
$\MM$, and we prove that
under mild restrictions on $\MM$, a random finite model of $\MM$ is almost
surely idemprimal. This implies that even if such an $\MM$ is distinguishable
from another idempotent linear
Maltsev condition by a finite model $\al{A}$ of $\MM$,
a random search for a finite model $\al{A}$ of $\MM$ with this property
will almost surely fail.
\end{abstract}

\maketitle

\section{Introduction}
\label{sec-intro}
This investigation arose from efforts by the first author
to construct random finite algebras that generate a variety
that is congruence $3$-permutable, but not $2$-permutable.
On the face of it, it ought to be easy to create such varieties.
The Hagemann--Mitschke terms provide a recipe for constructing a
$3$-permutable variety. By choosing the remaining values of those
operations randomly, one would expect that the
resulting variety would fail to satisfy any stronger identities (such as
$2$-permutability).  

It turns out that this is not the case: a random, finite, $3$-permutable
algebra almost surely generates a $2$-permutable variety.
Similar relationships can be sought among other Maltsev conditions.
If a random
finite algebra has a Maltsev term, will it have (with probability~1)
a majority term? Will a random finite algebra lying in a
congruence semidistributive variety almost surely generate one that is
congruence distributive?  

Now we want to make these questions and claims more precise.
As we will discuss below, there is a range of possible interpretations.
Our aim with the discussion is to clarify and motivate
our interpretation.
First recall that a strong Maltsev condition is a
condition of the form
\begin{equation}
  \label{eq-maltsevcond}
\text{``there exist terms $f_1,\dots,f_m$ which satisfy the identities in $\Sigma$''}
\end{equation}
where $\LL=\{f_1,\dots,f_m\}$ is a finite algebraic language and
$\Sigma$ is a finite set of $\LL$-identities.
We will refer to \eqref{eq-maltsevcond} as the Maltsev condition defined by
the pair $\MM=(\LL,\Sigma)$, and will denote it by $\CC_\MM$.
A variety $\var{V}$ satisfies the Maltsev condition $\CC_\MM$ in
\eqref{eq-maltsevcond} if and only if
there exist terms $f_1,\dots,f_m$ in the language of $\var{V}$ such that
for every member $\al{C}$ of $\var{V}$, the $\LL$-algebra
$(C;f_1^{\al{C}},\dots,f_m^{\al{C}})$ satisfies the identities in $\Sigma$.
Similarly, an algebra $\al{A}$ satisfies $\CC_\MM$ if and only if
there exist terms $f_1,\dots,f_m$ in the language of $\al{A}$ such that
the $\LL$-algebra 
$(A;f_1^{\al{A}},\dots,f_m^{\al{A}})$ satisfies the identities in $\Sigma$.
It is easy to see that $\al{A}$ satisfies $\CC_\MM$ if and only if the variety
generated by $\al{A}$ does.

A number of important properties of a variety are known to be characterized
by strong Maltsev conditions, including congruence permutability,
arithmeticity, and ``having a Taylor term''.
A rich supply of strong Maltsev conditions arise also in parametrized sequences
which define Maltsev conditions characterizing properties of varieties like
congruence distributivity and congruence modularity.
For further examples, see Section~\ref{sec-appl}. 
In most cases, the strong Maltsev conditions obtained in this way are
both linear and idempotent. 
Therefore,  we will assume throughout this paper that
$\MM$ is linear and idempotent (see the definitions in
Section~\ref{sec-il_ids}).
To avoid some degenerate cases, we will also assume
that $\MM$ is satisfiable (i.e., satisfiable by some algebra of size${}>1$).

Given a finite set $A$ and a finite algebraic language $\bar{\LL}$,
a random finite $\bar{\LL}$-algebra $\al{A}$ with universe $A$
can be obtained by
randomly filling out the tables of all operations $g^{\al{A}}$ interpreting
the symbols $g$ in $\bar{\LL}$. This way we get a discrete probability space
on the set of all $\bar{\LL}$-algebras with universe $A$, where all algebras
have the same probability.
For an abstract property $P$ of algebras (i.e., for a property that depends
only on the isomorphism types of algebras) we define the
\emph{probability that a finite $\bar{\LL}$-algebra has property $P$}
to be the limit, as $n\to\infty$, of the probability that
an $n$-element $\bar{\LL}$-algebra has property $P$
(see Section~\ref{sec-rmodels}).
This is the approach called \emph{labeled probability} by Freese~\cite{freese}.

Now let's return to the sort of questions alluded to in the opening paragraphs.
They have the form
\begin{enumerate}
\item[$(\dagger)$]
  ``How likely is it that
  a random finite algebra which satisfies the Maltsev condition
  $\CC_{\MM}$ will also satisfy [will fail to satisfy]
  the Maltsev condition $\CC_{\MM'}$?''
\end{enumerate}  
where $\MM$ and $\MM'$ are two systems of identities which are 
linear, idempotent, and satisfiable.

Question~$(\dagger)$, in its most straightforward interpretation, considers all random
finite algebras $\al{A}$ (in a specified finite language $\bar{\LL}$) which satisfy
$\CC_{\MM}$, and asks for the probability, in the sense discussed above, that such an
$\al{A}$ satisfies [fails to satisfy] $\CC_{\MM'}$.  By Murski\v{\i}'s
Theorem~\cite{murskii}, if $\bar{\LL}$ contains a symbol of arity${}\ge2$, then a random
finite $\bar{\LL}$-algebra is almost surely (i.e., with probability $1$)
\emph{idemprimal}.  Since the defining property of an idemprimal algebra is that its term
operations include all idempotent operations on its universe, such an algebra satisfies
every satisfiable strong, idempotent, linear Maltsev condition (see
Corollary~\ref{cor-models_exist}).  Thus, we get that a random finite algebra which
satisfies $\CC_{\MM}$ will satisfy $\CC_{\MM'}$ with probability $1$ [and hence will fail
to satisfy $\CC_{\MM'}$ with probability $0$], independently of the choice of $\MM$ and
$\MM'$.

This approach to question $(\dagger)$ is unsatisfying, not because the answer is trivial,
but because it is easy to pinpoint the weakness of this approach.  If one wants to
construct an algebra $\al{A}$ which satisfies a Maltsev condition $\CC_{\MM}$ and fails to
satisfy another, $\CC_{\MM'}$, then one would arrange the satisfaction of $\CC_{\MM}$ with
as few operations as possible; the more `unnecessary' operations are added to $\al{A}$,
the less likely it is that it will fail to satisfy $\CC_{\MM'}$.  The most `optimal'
satisfaction of $\CC_{\MM}$ is achieved by requiring that $\al{A}$ satisfies $\CC_{\MM}$
with its basic operations (and $\al{A}$ has no other basic operations).  This is
equivalent to requiring that, if $\MM=(\LL,\Sigma)$, then $\al{A}$ is an $\LL$-algebra
satisfying the identities in $\Sigma$.  Such an algebra will be called a \emph{model of
  $\MM$}.

This motivates us to adopt the following interpretation
of the kind of questions alluded to at the
beginning of the introduction:
\begin{enumerate}
\item[$(\ddagger)$]
  ``How likely is it that
  a random finite model of $\MM$ 
  will satisfy [will fail to satisfy]
  the Maltsev condition $\CC_{\MM'}$?''
\end{enumerate}  
where $\MM$ and $\MM'$ are two finite systems of identities which are 
linear, idempotent, and satisfiable.
However, it should be noted that 
there is a price to pay for considering only random models of $\MM$ instead of
random algebras satisfying the corresponding Maltsev condition $\CC_\MM$: 
in Section~\ref{sec-appl} 
we will see examples showing there exist 
different systems $\MM_1,\MM_2$  that  describe equivalent Maltsev
conditions, but the random finite models of $\MM_1$ and the random finite
models of $\MM_2$ have essentially different properties, and hence
question~$(\ddagger)$ might have different answers for $\MM=\MM_1$ and
$\MM=\MM_2$. 

Our main result in this paper, Theorem~\ref{thm-idempr},
is a characterization of those
finite, satisfiable systems $\MM$ of idempotent linear identities
for which a random finite model is almost surely idemprimal.
As shown in Section~\ref{sec-appl}, these include many of the familiar
strong Maltsev conditions.
For these systems $\MM$, question~$(\ddagger)$ has the same
answer as question~$(\dagger)$:
A random finite model of $\MM$ almost surely satisfies every satisfiable strong,
idempotent, linear Maltsev condition $\CC_{\MM'}$.

To obtain this result, we first study syntactic properties of 
finite systems $\MM$ of idempotent linear identities (Section~\ref{sec-il_ids}),
which we then apply to analyze random finite models of $\MM$
(Section~\ref{sec-rmodels}).
We show in Theorem~\ref{thm-alg-descr} how one can construct all finite models
of $\MM$ by randomly (and independently)
filling out well-chosen parts of the operation tables,
and then by completing the remaining parts of the tables
(which are uniquely determined) so that all required identities are satisfied.
This result is crucial for the counting arguments we need for
determining the probability of idemprimality for random models of $\MM$.
In Section~\ref{sec-subalg} we prove that for every
finite, satisfiable systems $\MM$ of idempotent linear identities
the random models of $\MM$ have only small proper subalgebras, where
`small' depends on some parameters of $\MM$.
Section~\ref{sec-idempr} is devoted to the proof of our main result,
and in the final Section~\ref{sec-appl} we apply our results to some familiar
strong idempotent linear Maltsev conditions, and answer the specific questions
in the opening paragraphs above.

\section{Systems of Idempotent Linear Identities}
\label{sec-il_ids}

Let $\LL$ be an algebraic language with no constant symbols, and let
$V:=\{v_1,v_2,\dots\}$ be a set of distinct variables indexed by
positive integers.
We form $\LL$-terms using these variables only, but for convenience, we may use
other notation like $x,y,z$ or $x_i,y_j, z_k$, etc. for these variables.
An $\LL$-term is called \emph{linear} if it contains at most one operation
symbol.
An $\LL$-identity or a set of $\LL$-identities is called \emph{linear}
if all terms involved are linear.
If $s$ is an $\LL$-term and $X$ is a set of variables that contains all
variables occurring in $s$, then for any function $\gamma\colon X\to X$,
$s[\gamma]$ will denote the term obtained from $s$ by
replacing each variable $x\in X$ with $\gamma(x)\in X$.
In the special case in which $X$ is precisely the set of variables occurring in $s$,
a term of this form $s[\gamma]$ will be called an
\emph{identification minor} of $s$. In particular, $s[\gamma]$ is a
\emph{proper identification minor} of $s$ if $\gamma$ is not
injective.
Thus, the linear $\LL$-terms are the variables and all terms of the form $f[\gamma]$ in
which $f\in \LL$ and $\gamma\colon X \to X$ for some $X\subseteq V$ containing
$\{v_1,\dots,v_{\arity(f)}\}$. In this context we also use the symbol $f$ as shorthand for
the term $f(v_1,\dots,v_{\arity(f)})$. 
For any set $X$ of variables,
the set of all linear $\LL$-terms with variables in $X$
will be denoted by $\LT^\LL_X$.

As usual, if $\Sigma\cup\{\phi\}$ is a set of identities in the language $\LL$, we write
$\Sigma\models\phi$ to denote that every $\LL$-algebra that satisfies all identities in
$\Sigma$ also satisfies $\phi$.  We may call two terms, $s$ and $t$,
\emph{$\Sigma$-equivalent} if $\Sigma\models s\approx t$.  We will say that a set
$\Sigma$ of $\LL$-identities is \emph{unsatisfiable}, if it has no model of size greater
than~1, or equivalently, if $\Sigma\models x\approx y$ for distinct variables $x,y$;
otherwise we will say that $\Sigma$ is \emph{satisfiable}.
If $\Sigma$ is a set of linear $\LL$-identities, we might call the pair
$\MM=(\LL,\Sigma)$ a \emph{linear system}.

If $\Sigma\cup{\phi}$ is a set of linear identities, then there is a simple syntactic
characterization for the relation $\Sigma\models\phi$, due to David Kelly~\cite{kelly},
which we state in Theorem~\ref{thm-kelly} below\footnote{%
  Kelly's Theorem is slightly more general than Theorem~\ref{thm-kelly}: it allows
  constant symbols in $\LL$ and in the substitutions $\gamma$.  The theorem is a
  restriction of Birkhoff's completeness theorem for equational logic to the set of
  so-called `basic' identities, but it shows that there is a simple algorithm that decides
  for every set $\Sigma\cup{\phi}$ of basic identities whether $\Sigma\models\phi$.},
after introducing some terminology and notation.

Let $\MM=(\LL,\Sigma)$ be a linear system. For any set $X$ of variables, let
\[
\EE^{\MM}_X:=
\{(s,t)\in\LT^{\LL}_X\times\LT^\LL_X:\Sigma\models s\approx t\},
\]
the restriction of `$\Sigma$-equivalence of terms' to $\LT^{\LL}_X$. 
Clearly, this is an equivalence relation on $\LT^{\LL}_X$. 

If $X$ is a set of variables that contains all variables occurring in
$\Sigma$, we will use $\equiv^\MM_X$ to denote the least equivalence relation
on $\LT^\LL_X$ satisfying the following conditions:
\begin{enumerate}
\item[(i)]
  $\equiv^\MM_X$ contains $\Sigma$, i.e., $s\equiv^\MM_X t$ for every identity
  $s\approx t$ in $\Sigma$, and
\item[(ii)]
  $\equiv^\MM_X$ is closed under substitutions of variables, i.e.,
  whenever $s\equiv^\MM_X t$ holds for some $s,t\in\LT^\LL_X$, we also have
  $s[\gamma]\equiv^\MM_X t[\gamma]$ for all functions
  $\gamma\colon X\to X$.   
\end{enumerate}  
Thus, the pairs in $\equiv^\MM_X$ are obtained from the identities
in $\Sigma$ by applying the closure conditions induced by reflexivity,
symmetry, transitivity, and the condition in (ii). To emphasize that
these closure conditions are simple rules of inference for identities, we 
may write $\Sigma\vdash_X s\approx t$ to indicate that $s\equiv^\MM_X t$.

Finally,
we will say that $X$ is \emph{large enough for $\Sigma$} (or for $\MM$)
if
\begin{itemize}
\item
  $|X|\ge2$ and $X$ contains all variables occurring in $\Sigma$,
\item
  $|X|\ge\arity(f)$ for every $f\in\LL$, and
\item
  $|X|$ is at least as large as the number of distinct variables occurring in
  each identity in $\Sigma$.
\end{itemize}

\begin{thm}[Kelly~\cite{kelly}; for a published proof, see \cite{KKSz-rates1}]
\label{thm-kelly}
Let $X$ be a set of variables, and let
$\MM=(\LL,\Sigma)$ be a linear system in an
algebraic language $\LL$ without constant symbols.
\begin{enumerate}
\item
  If $X$ is large enough for $\Sigma$, then 
  $\Sigma$ is unsatisfiable if and only if $\Sigma\vdash_X x\approx y$
  for distinct variables $x,y\in X$.
\item
  Assume $\Sigma$ is satisfiable and $s\approx t$ is a linear
  $\LL$-identity.
  If $X$ is large enough for $\Sigma\cup\{s\approx t\}$, then for any
  $s,t\in\LT^\LL_X$ we have that
  \[
  \Sigma\models s\approx t
  \qquad\text{iff}\qquad
  \Sigma\vdash_X s\approx t.
  \]
\end{enumerate}
\end{thm}  

Let $\MM=(\LL,\Sigma)$ be linear and $X$ large enough for $\MM$.
The equivalence class of a term $t\in\LT^\LL_X$
in the equivalence relation $\EE^\MM_X$ (or, equivalently, $\equiv^\MM_X$)
will be denoted by $\EE^\MM_X[t]$. If $X$, $\LL$, $\MM$ are clear
from the context, we may write $\LT$, $\EE$, and $\equiv$ for
$\LT^\LL_X$, $\EE^\MM_X$, and $\equiv^\MM_X$, respectively.

From now on we will focus on the case when $\MM=(\LL,\Sigma)$
is also \emph{idempotent}, that is,
$\Sigma\models f(x,\dots,x)\approx x$ holds
for all operation symbols $f$ in $\LL$.
Lemmas~\ref{lm-ess-vars}--\ref{lm-symm} below establish basic properties
of the equivalence relation $\EE^\MM_X$ which will be used
later on in the paper. 

\begin{lm}
  \label{lm-ess-vars}
Assume $\MM=(\LL,\Sigma)$ is idempotent, linear, and satisfiable, and
let $X$ be a large enough set of variables for $\MM$.  
For every equivalence class $C$ of $\EE=\EE^\MM_X$ there exists
a unique nonempty set $X_C\subseteq X$
such that
\begin{itemize}
\item
  $X_C$ is the set of variables of some term in $C$, and
\item
  every variable in $X_C$ occurs in all terms in $C$.
\end{itemize}
Moreover, we have that
\begin{itemize}
\item
  the terms $t\in C$ are independent, relative to $\Sigma$, of their variables
  not in $X_C$; i.e., $\Sigma\models t(X_C,\bar{z})\approx t(X_C,\bar{z}')$
  for arbitrary lists of variables $\bar{z},\bar{z}'$ in $X$.
\end{itemize}  
\end{lm}  

\begin{proof}
  Let $C$ be any equivalence class of $\EE$, and let $s\in C$ be such that
  the set $V_s$ of variables occurring in $s$ is minimal (with respect to
  $\subseteq$) among the members of
  $C$. For the first claim
  it suffices to show that all variables in $V_s$ occur in every member
  of $C$. Suppose not, and let $t\in C$ be a witness to this fact.
  Thus, the set $V_t$ of variables occurring in $t$ is incomparable to $V_s$.
  It cannot be that $V_s\cap V_t=\emptyset$,
  because then $(s,t)\in\EE$ --- i.e., $\Sigma\models s\approx t$ ---
  would imply $\Sigma\models s(x,\dots,x)\approx t(y,\dots,y)$
  for any distinct $x,y\in X$.  
  By idempotence this would yield $\Sigma\models x\approx y$, contradicting
  our assumption that $\Sigma$ is satisfiable. Thus $V_s\cap V_t\not=\emptyset$.
  Let $s=s(\bar{x},\bar{y})$ and $t=t(\bar{x},\bar{z})$ where
  $\bar{x}=(x_1,\dots,x_\ell)$ ($\ell\ge1$)
  lists the variables in $V_s\cap V_t$, and
  $\bar{y}$, $\bar{z}$ list the variables in $V_s\setminus V_t$ and
  $V_t\setminus V_s$ respectively. By the choice of $s$ and $t$,
  $\bar{y}$ and $\bar{z}$ are both nonempty.
  Now $(s,t)\in\EE$ --- or equivalently,
  $\Sigma\models s(\bar{x},\bar{y})\approx t(\bar{x},\bar{z})$ ---
  implies that
  $\Sigma\models s(\bar{x},\bar{y})\approx t(\bar{x},x_1,\dots,x_1)
  \approx s(\bar{x},x_1,\dots,x_1)$.
  Thus, $s(\bar{x},x_1,\dots,x_1)\in C$ and the set of variables occurring
  in this term, $V_s\cap V_t$, is a proper subset of $V_s$. This contradicts the
  minimality property of $s$, and hence proves the first claim.

  For an arbitrary $t\in C$ the same argument as
  above shows that $t=t(X_C,\bar{y})$
  for some (possibly empty) list $\bar{y}$ of variables which is disjoint
  from $X_C$, and
  $\Sigma\models t(X_C,\bar{y})\approx t(X_C,x_1,\dots,x_1)$
  for any $x_1\in X_C$.
  Thus,
  \[
  \Sigma\models t(X_C,\bar{z})\approx t(X_C,x_1,\dots,x_1)
  \approx t(X_C,\bar{z}')
  \]
  for arbitrary lists of variables $\bar{z},\bar{z}'$ in $X$, as claimed.
\end{proof}

Under the same assumptions on $\MM$ and $X$ as in Lemma~\ref{lm-ess-vars}, 
it is easy to see that
the equivalence relation $\EE=\EE^\MM_X$ is invariant
under all permutations of the variables in $X$;
that is, for every permutation $\gamma\in S_X$ and for
arbitrary terms $s,t\in\LT=\LT^\LL_X$, we have
    \[
    (s,t)\in\EE\qquad\text{if and only if}\qquad
    (s[\gamma],t[\gamma])\in\EE.
    \]
Hence, the symmetric group $S_X$ has an induced action on the set of
blocks of $\EE$ defined by
    \[
    \gamma\cdot\EE[t]=\EE[t[\gamma]]
    \qquad
    \text{for all $t\in\LT$ and $\gamma\in S_X$}.
    \]

\begin{lm}
  \label{lm-symm}
  Assume $\MM=(\LL,\Sigma)$ is idempotent, linear, and satisfiable, and
  let $X$ be a large enough set of variables for $\MM$.
The following statements hold for
arbitrary equivalence classes $B$ and $C$ 
of
$\EE=\EE^\MM_X$:
    \begin{itemize}
    \item
      If $B$ and $C$ are in the same orbit of $\EE$ 
      under the action of $S_X$, say $\gamma\cdot B=C$, then
      $\gamma$ restricts to a bijection $X_B\to X_C$, and
      $\gamma'\cdot B=C$ for all $\gamma'\in S_X$ satisfying
      $\gamma'(x)=\gamma(x)$ for all $x\in X_B$.
    \item
      In particular,
      $S_{X_C}$ has a unique subgroup $G_C$
      such that
      for any $\gamma\in S_X$ we have $\gamma\cdot C=C$ if and only if
      $\gamma(X_C)=X_C$ and $\gamma\restr{X_C}\in G_C$.
    \item
      Moreover, for any $t\in C$ whose variables are exactly the variables in
      $X_C$, and for any permutation $\pi\in S_{X_C}$,
      \[
      \Sigma\models t\approx t[\pi]
      \qquad\text{iff}\qquad
      (t,t[\pi])\in\EE
      \qquad\text{iff}\qquad
      \pi\in G_C.
      \]
    \end{itemize}    
\end{lm}  

\begin{proof}
  All three claims are straightforward consequences of
  Lemma~\ref{lm-ess-vars} and the
  definition of the action of $S_X$.
\end{proof}  

\begin{df}
  \label{df-ess-vars_symm}
Let $\MM=(\LL,\Sigma)$ be idempotent, linear, and satisfiable, and
let $X$ be a large enough set of variables for $\MM$.
For any equivalence class $C$ of $\EE=\EE^\MM_X$,
\begin{itemize}
  \item
we will refer to
the set $X_C$ (see Lemma~\ref{lm-ess-vars})
as \emph{the set of essential variables} of the terms in $C$,
and if $t\in C$, we may write $X_t$ in place of $X_C$, and call it
\emph{the set of essential variables} of $t$;
\item
we will refer to the group $G_C$ (see Lemma~\ref{lm-symm})
as \emph{the symmetry group} of the terms in $C$, and
if $t\in C$, we may write $G_t$ for $G_C$ and call it
\emph{the symmetry group} of $t$.
\end{itemize}
If we want to emphasize the dependence of these notions on $\MM$ we may
talk about essential variables or symmetry groups of terms
\emph{relative to $\MM$}. 
\end{df}

\begin{exmp}
  \label{exmp:cmaltsev}
Let $\LL$ consist of a single ternary operation symbol, $f$, and take
$\MM=(\LL,\Sigma)$ where
\begin{equation*}
  \Sigma=\{f(x,x,y)\approx y,\,  f(x,y,z) \approx f(z,y,x)\}.
\end{equation*}
Since $\Sigma\models f(x,x,x)\approx x$, we see that $\MM$ is idempotent
and linear. It is clear from the definition that the set $X=\{x,y,z\}$
of variables is large enough for $\MM$.
There are a total of 30 linear terms
in $\LT=\LT^\LL_X$:
$27$ terms containing
$f$, plus the $3$ variables.
One can easily determine the equivalence relation $\EE=\EE^\MM_X$, using
Theorem~\ref{thm-kelly}; it turns out that
$\EE$ partitions $\LT$ as shown in the first two columns of
Table~\ref{fig:cmaltsev}.
Hence, in particular, it follows that $\MM$ is satisfiable.

\begin{table}
  \centering
  \renewcommand{\arraystretch}{1.15}
  \begin{tabular}{|l|l|l|l|}
    \hline
    $C$ & \hfil Members \hfil& $X_C$ & $G_C\,(\le S_{X_C})$\\
    \hline
    $B_1$ & $x$, $f(x,x,x)$, $f(y,y,x)$, $f(x,y,y)$, $f(z,z,x)$,
            $f(x,z,z)$ & $\{x\}$ & $\{\id\}$\\
    $B_2$ & $y$, $f(y,y,y)$, $f(x,x,y)$, $f(y,x,x)$, $f(z,z,y)$,
            $f(y,z,z)$ & $\{y\}$ & $\{\id\}$\\
    $B_3$& $z$, $f(z,z,z)$, $f(x,x,z)$, $f(z,x,x)$, $f(y,y,z)$, $f(z,y,y)$
                  &$\{z\}$ & $\{\id\}$\\
    \hline
    $C_1$ & $f(x,y,x)$ & $\{x,y\}$ & $\{\id\}$\\
    $C_2$ & $f(y,x,y)$ & $\{x,y\}$ & $\{\id\}$\\
    $C_3$ & $f(x,z,x)$ & $\{x,z\}$ & $\{\id\}$\\
    $C_4$ & $f(z,x,z)$ & $\{x,z\}$ & $\{\id\}$\\
    $C_5$ & $f(y,z,y)$ & $\{y,z\}$ & $\{\id\}$\\
    $C_6$ & $f(z,y,z)$ & $\{y,z\}$ & $\{\id\}$\\
    \hline
    $D_1$ & $f(x,y,z)$, $f(z,y,x)$ & $\{x,y,z\}$ & $\langle(x\ z)\rangle$\\
    $D_2$ & $f(y,x,z)$, $f(z,x,y)$ & $\{x,y,z\}$ & $\langle(y\ z)\rangle$\\
    $D_3$ & $f(x,z,y)$, $f(y,z,x)$ & $\{x,y,z\}$ & $\langle(x\ y)\rangle$\\
    \hline
  \end{tabular}
  \smallskip
  \caption{Equivalence classes under $\EE$ for Example~\ref{exmp:cmaltsev}}\label{fig:cmaltsev}
\end{table}

The
action of the symmetric
group $S_X$ induces three orbits on the blocks
of $\EE$:
$\{B_1,B_2, B_3\}$, $\{C_1,C_2,C_3,C_4,C_5,C_6\}$, and $\{D_1,D_2,D_3\}$.
The symmetry
groups of the blocks $B_i$ and $C_i$ are all trivial. The symmetry
group of each $D_i$ has order~2. For example, $G_{D_1}$ contains the
permutation transposing $x$ and $z$.
The last two columns of Table~\ref{fig:cmaltsev} display
the sets of essential variables and the symmetry groups 
of all equivalence classes of $\EE$.
\qeddiamond
\end{exmp}

Returning to our general discussion, let again $\MM$ and $X$ be as
in Lemmas~\ref{lm-ess-vars}--\ref{lm-symm}, and for any positive integer
$n$, let $[n]$ denote the set $\{1,\dots,n\}$.
Consider a linear term $r=r(x_1,\dots,x_k)$ in $\LT=\LT_X^\LL$ 
where the variables $x_1,\dots,x_k\in X$ are distinct.
We have $r(x_1,\dots,x_k)=f(x_{\phi(1)},\dots,x_{\phi(d)})$ for some
$f\in\LL$ with  $d=\arity(f)$ and some onto function
$\phi\colon[d]\to[k]$. It is easy to see that if 
$f(y_1,\dots,y_d)\in\LT$ is another linear term, then
under the action of $S_X$ by permuting variables,
$f(y_1,\dots,y_d)\in\LT$ lies in the same orbit as $r$
--- that is,
$f(y_1,\dots,y_d)=r[\gamma]$ for some $\gamma\in S_X$ --- if and only if
the $d$-tuples of variables
$(x_{\phi(1)},\dots,x_{\phi(d)})$ and $(y_1,\dots,y_d)$
have the same `pattern' in the following sense.

\begin{df}
For any set $U$ and $d$-tuple $\bar{u}=(u_1,\dots,u_d)\in U^d$
we define the \emph{pattern of $\bar u$} to be
the equivalence relation
$\epsilon(\bar u):=\{\,(i,j)\in[d]^2 : u_i=u_j\,\}$
on $[d]$.
Two $d$-tuples, $\bar u,\bar v\in U^d$ are said to
\emph{have the same pattern} if 
$\epsilon(\bar u) = \epsilon(\bar v)$.
We might refer to equivalence relations on $[d]$ as
\emph{patterns on $[d]$}.
Given a pattern $\mu$ on $[d]$, we define
\begin{equation*}
  U^{(\mu)} = \{\bar u\in U^d : \epsilon(\bar u)=\mu\}.
\end{equation*}
We shall write $U^{(d)}$ in place of $U^{(\mu)}$ when
$\mu$ is the equality relation on $[d]$. 
\end{df}

In this terminology, the pattern of the variables
$(x_{\phi(1)},\dots,x_{\phi(d)})$ of the term $r=f(x_{\phi(1)},\dots,x_{\phi(d)})$
is the kernel of the function $\phi$.
Furthermore, if a term, $s\in\LT$, exhibits a pattern $\mu$ in its variables,
then for any $\gamma \in S_X$, the term $s[\gamma]$ also has pattern
$\mu$. Thus, if a term appears in an equivalence class, $C$, of $\EE$,
then a similar term, with the same pattern of variables, appears in every
block of the orbit of $C$, and in no other orbits. 

We now introduce another concept which will
be useful in the forthcoming sections, and is related  
to the equivalence classes of $\EE$ and the orbits of the action of
$S_X$ on the set of equivalence classes.

\begin{df}
Let $\MM=(\LL,\Sigma)$ be idempotent, linear, and satisfiable, and
let $X$ be a large enough set of variables for $\MM$.
We will say that two terms $s,t\in\LT=\LT^\LL_X$ are
\emph{essentially different for $\MM$} if
in the equivalence relation $\EE=\EE^\MM_X$, 
the equivalence classes $\EE[s]$ and $\EE[t]$ of $s$ and $t$
belong to different orbits of $S_X$.
\end{df}

Equivalently, $s$ and $t$ are essentially different for $\MM$
if and only if $\Sigma\not\models s[\gamma]\approx t[\delta]$ for any
$\gamma,\delta\in S_X$.
Lemma~\ref{lm-symm} implies that, if
$s,t\in\LT$ have different numbers of essential
variables (relative to $\MM$), then $s,t$ are essentially different
for $\MM$. Furthermore, if
$s=s(x_1,\dots,x_d)\in\LT$ and
$t=t(x_1,\dots,x_d)\in\LT$ are two linear terms
such that $\{x_1,\dots,x_d\}$ is the set of essential variables of both,
then they are essentially different for $\MM$ if and only if
$\Sigma\not\models s(x_1,\dots,x_d)\approx t\bigl(\pi(x_1),\dots,\pi(x_d)\bigr)$
for any permutation $\pi$ of $\{x_1,\dots,x_d\}$.

\bigskip

In the remainder of this section we will discuss minimal terms for $\MM$,
which we now define.

\begin{df}
Let $\MM=(\LL,\Sigma)$ be idempotent and linear, and let
$t=t(x_1,\dots,x_d)$ be a linear $\LL$-term (where the variables
$x_1,\dots,x_d$ are distinct).
We will say that
\begin{enumerate}
\item[(1)]
$t$ is a \emph{trivial term for $\MM$} if
$\Sigma\models t\approx z$ for some variable $z$, and
\item[(2)]
$t$ is a \emph{minimal term for $\MM$} if
  \begin{itemize}
  \item
    $\Sigma\not\models t\approx z$ for any variable $z$, but
  \item
    $\Sigma\models t[\gamma]\approx z_\gamma$ for some variable $z_\gamma$,
    whenever $\gamma\colon\{x_1,\dots,x_d\}\to\{x_1,\dots,x_d\}$
    is a non-injective function.
  \end{itemize}
\end{enumerate}  
In other words,
$t$ is a trivial term for $\MM$ iff $t$ is
  $\Sigma$-equivalent to a variable, while
$t$ is a minimal term for $\MM$ iff $t$ is non-trivial,
  but every proper identification minor of 
  $t$ is trivial.
\end{df}  

Notice that
a minimal term exists for $\MM$ only if $\MM$ is satisfiable. Moreover,
if $t=t(x_1,\dots,x_d)$ is a minimal term for $\MM$,
then $X_t=\{x_1,\dots,x_d\}$, that is, all variables $x_1,\dots,x_d$ of
$t$ are essential.

\begin{thm}
  \label{thm-min-t}
    Assume $\MM=(\LL,\Sigma)$ is idempotent, linear, and satisfiable.
  For every $f\in\LL$, 
    either the term $f=f(x_1,\dots,x_{\arity(f)})$
  \textup(where $x_1,\dots,x_{\arity(f)}$ are distinct variables\textup)
  is trivial for $\MM$, or it
    has an identification minor that is a 
  minimal term for $\MM$.
    Moreover, every minimal term  $t$ of $\MM$ satisfies
  one of the following conditions, up to a permutation of its
  variables:
  \begin{enumerate}
\item[{\rm(1)}]
  $t$ is
    a nontrivial binary term for $\MM$,
    that is,
  $t=t(x,y)$ such that
  \[
  \Sigma\not\models t\approx x\quad\text{and}\quad
  \Sigma\not\models t\approx y.
  \]
\item[{\rm(2)}]
  $t$ is a minority term
    for $\MM$,
    that is, $t=t(x,y,z)$ with
  \[
  \Sigma\models t(x,y,y)\approx t(y,x,y)\approx t(y,y,x)\approx x.
  \]
\item[{\rm(3)}]
  $t$ is a $\frac{2}{3}$-minority term
    for $\MM$,
    that is, $t=t(x,y,z)$ with
  \[
  \Sigma\models t(x,y,y)\approx t(x,y,x)\approx t(y,y,x)\approx x.
  \]
\item[{\rm(4)}]
  $t$ is a majority term
    for $\MM$,
    that is, $t=t(x,y,z)$ with
  \[
  \Sigma\models t(x,y,y)\approx t(y,x,y)\approx t(y,y,x)\approx y.
  \]
\item[{\rm(5)}]
  $t$ is a semiprojection term
    for $\MM$,
    that is,
  $t=t(x_1,\dots,x_d)$ $(d\ge3)$,
  \begin{align*}
    \Sigma &{}\not\models t(x_1,\dots,x_d)\approx x_1, \quad\text{but}\\
    \Sigma &{}\models t(y_1,\dots,y_d)\approx y_1\quad
    \text{for any variables $y_1,\dots,y_d$ with
      $|\{y_1,\dots,y_d\}|<d$.}
  \end{align*}
\end{enumerate}
\end{thm}  

\begin{proof}
Let $f\in\LL$ 
be such that the term $f=f(v_1,\dots,v_{\arity(f)})$ is
nontrivial for $\MM$.
To see that $f$ has
an identification minor that is a minimal term for $\MM$,
let $X$ be a set of variables that is large enough for $\MM$
and contains $v_1,\dots,v_{\arity(f)}$,
and consider
all identification minors $t=f[\gamma]\in\LT$ of
the term $f=f(v_1,\dots,v_{\arity(f)})$.
These include the term $t=f$, which is nontrivial
for $\MM$ by assumption, and
they also include the terms
$t=f(x,\dots,x)$ ($x\in X$)
which are trivial for $\MM$,
as $\MM$ is idempotent.
Therefore, among all identification minors of $f$, there exists
a term $t$ such that $t$ is nontrivial for $\MM$
and $d=|X_t|$ is as small as possible.
Clearly, $d>1$.
We may assume without loss of generality that $X_t$ is exactly the set
of variables that occur in $t$.
Thus, $t=t(x_1,\dots,x_d)$ with $\{x_1,\dots,x_d\}=X_t$.
By the choice of $t$, the conditions defining
a minimal term for $\MM$
hold for $t$.

Now let $t=t(x_1,\dots,x_d)$ be any minimal
term for $\MM$.
Then $\Sigma\not\models t(x_1,\dots,x_d)\approx z$ for any variable $z$, but
$\Sigma\models t(y_1,\dots,y_d)\approx y$
for some variable $y$ whenever $|\{y_1,\dots,y_d\}|<d$.
Were $y\notin \{y_1,\dots,y_d\}$, we would get (by substituting $z$ for $y$)
$\Sigma\models y\approx z$ for distinct variables $y,z$, which contradicts
our assumption that $\MM$ is
satisfiable.
Thus $y=y_i$ for some $i$.
If $d\le3$, then the only possibilities,
up to permutations of variables,
are those described in (1)--(5).
If $d\ge4$, then by \'{S}wierczkowski's Lemma \cite{swier}, the only
possibility,
up to permutations of variables,
is (5).
\end{proof}

\section{Random Models}
\label{sec-rmodels}

Let $\MM=(\LL,\Sigma)$ where $\Sigma$ is a finite system of idempotent
linear identities in a finite language $\LL$.
For every finite set $A$, let $\ModA(\MM)$ denote the set of all
models of $\MM$ on~$A$. Clearly, $\ModA(\MM)$ is a finite set.
Motivated by our discussion in the introduction, we will
stipulate that every member of $\ModA(\MM)$ has the same probability,
so we get a discrete probability space on $\ModA(\MM)$
with uniform distribution.
Accordingly, for every property $P$ of algebras on $A$, the probability
that a random model $\al{A}$ of $\MM$ on $A$ has property $P$ is
\begin{equation}
  \label{eq-probabA}
\PrA(\text{$\al{A}$ has property $P$}):=
\frac{|\{\al{A}\in\ModA(\MM):\text{$\al{A}$ has property $P$}\}|}
     {|\ModA(\MM)|}.
\end{equation}
Note that we add a subscript $A$ to $\Pr$ to indicate the base set of the
models we are considering. In contrast, $\MM$ is suppressed
in the notation, because $\MM$ will usually be fixed and clear from the context
when we apply the notation $\Pr_A$.

We will call a property $P$ of algebras an \emph{abstract property}
(of finite algebras) if for any two isomorphic (finite)
algebras $\al{A}$ and $\al{B}$, $\al{A}$ has property
$P$ if and only if $\al{B}$ does.
It follows that if $A$, $B$ are finite sets of the same cardinality and
$P$ is an abstract property,
then in the probability spaces of all models of $\MM$
on $A$ and $B$, respectively, we have
\[
\PrA(\text{$\al{A}$ has property $P$})=
\PrB(\text{$\al{B}$ has property $P$}).
\]
Therefore, our main concept below is well-defined.

\begin{df}
  \label{df-probab}
  Let $P$ be an abstract property of finite algebras.
  We will say that \emph{a random finite model of $\MM$ has property $P$
  with probability $p$} if
\begin{equation}
  \label{eq-probab}
  p=\lim_{n\to\infty}\Prn(\text{$\al{A}$ has property $P$}).
\end{equation}
\end{df}  

Clearly, this limit is not affected by disregarding the
values of $\Prn(\text{$\al{A}$ has property $P$})$
on the right hand side for finitely many $n$'s.
Therefore, when computing these probabilities,
we may, and we often will, restrict
to models $\al{A}$ of $\MM$ whose universe
$[n]$ has large enough cardinality.

If every linear $\LL$-term is trivial for $\MM$ (i.e., $\Sigma$-equivalent to a
variable), then either $\MM$ is satisfiable and
$\MM$ has exactly one model of the form $\al{A}=\langle [n];\LL\rangle$
for every positive integer $n$, or $\MM$ is unsatisfiable and
$\MM$ has exactly
one model of the form $\al{A}=\langle [n];\LL\rangle$ for $n=1$ and none for
$n>1$.
Hence, every probability on the right hand side of \eqref{eq-probab}
is $0$ or $1$. This degenerate case is not interesting, and
excluding it from our considerations will make our theorems easier to state.
Therefore, for the rest of the paper we adopt the following assumption
on $\MM$:

\begin{glassumption}
  \label{global-assumption}
  We assume that $\MM=(\LL,\Sigma)$ where
  \begin{itemize}
  \item
    $\LL$ is a finite algebraic language and
  \item
    $\Sigma$ is a finite system of idempotent linear $\LL$-identities such that
  \item
    there exists a nontrivial linear $\LL$-term for $\MM$.
  \end{itemize}
\end{glassumption}

In this section our goal is to analyze the finite models of $\MM$, and show
how the operations of each such model can be constructed from a family
of independently chosen functions with certain symmetry properties.
The significance of this result is twofold: (i) it will provide an algorithm
for choosing, with probability $\frac{1}{|\Mod_A(\MM)|}$,
a random model of $\MM$ on a fixed finite set $A$; (ii) it will allow us
to do counting arguments to find probabilities of the form~\eqref{eq-probabA}.

The intuitive idea for this description of the finite models of $\MM$ is
quite simple. 
Every linear $\LL$-term $t$ induces a term operation
$t^{\al{A}}$ on the universe $A$ of every model $\al{A}$ of $\MM$,
which can be further restricted to subsets of the domain of $t^{\al{A}}$.
It turns out that for a well-chosen family $(t_i)$ of terms,
which depends on $\MM$,
the term operations $t_i^{\al{A}}$ of the models $\al{A}$ of $\MM$
restrict to appropriately chosen subsets
of their domains as independent functions $h_i$ with certain symmetries.
Conversely, any family $(h_i)$ of independently chosen
functions with these symmetry properties (see Definition~\ref{df-mfam})
gives rise to a model of $\MM$.
Unfortunately, nailing down all the details in complete generality is
somewhat tedious, therefore we start by illustrating the method with
an example.

\begin{exmp}
  \label{exmp:cmaltsev2}
  Let $\MM=(\LL,\Sigma)$, $X$, and $\EE$ 
  be as defined in Example~\ref{exmp:cmaltsev}, and let
  $A$ be a nonempty set. We want to describe how to construct all models
  $\al{A}=\langle A;f^{\al{A}}\rangle$ of $\MM$, equivalently,
  how to construct all operations $f^{\al{A}}$ on the set
  $A$ which obey the identities in $\Sigma$.
  We will use Table~\ref{fig:cmaltsev}, which shows the blocks
  $B_1,\dots,D_3$ of $\EE$,
  split into orbits $\{B_1,B_2,B_3\}$, $\{C_1,\ldots,C_6\}$, and
  $\{D_1,D_2,D_3\}$ of $S_X$.
  From this, we can read off all linear identities $r\approx s$
  in the variables $x,y,z$ which are consequences of
  $\Sigma$; namely, $r\approx s$ is such an identity if and only if $r$ and $s$
  are in the same block of $\EE$ (i.e., appear in the same row of the table).
  Each one of these identities forces the desired operation
  $f^{\al{A}}$ to satisfy a condition of the following form:
  \begin{itemize}
    \item
 ``$f^{\al{A}}$ applied to a triple $(a,b,c)\in A^3$ of some pattern
  has to equal $f^{\al{A}}$ applied to a triple $(a',b',c')\in A^3$ of
  another pattern'', or
    \item
  ``$f^{\al{A}}$ applied to a triple $(a,b,c)\in A^3$ of some pattern
  has to equal one of the arguments''.
  \end{itemize}    
  Identities that come from blocks of $\EE$ in the same orbit of $S_X$
  contain the same information, therefore we will choose and fix a transversal
  $\mathcal{C}$ for the orbits of $S_X$; say $\mathcal{C}:=\{B_1,C_1,D_1\}$.
  Let us also choose and fix a representative from each of these
  blocks; say we choose the term $t_{B_1}=t_{B_1}(x):=x$ from $B_1$,
  $t_{C_1}=t_{C_1}(x,y):=f(x,y,x)$
  from $C_1$, and $t_{D_1}=t_{D_1}(x,y,z):=f(x,y,z)$ from $D_1$.
  Thus, $\{t_{B_1},t_{C_1},t_{D_1}\}$ is a maximal family of essentially
  different linear terms for $\MM$.
  Notice also that $t_{B_1},t_{C_1},t_{D_1}$
  were chosen so that all of their variables are essential.
  
  First, we want to deduce some necessary conditions for $f^{\al{A}}$
  to obey the identities in $\Sigma$. So, suppose $f^{\al{A}}$ obeys the
  identities in $\Sigma$. Then it also obeys all identities that come from the
  blocks $B_1,C_1,D_1$. We can use these identities to express
  $f^{\al{A}}(a,b,c)$, for triples $(a,b,c)\in A^3$ of various patterns,
  in terms of
  three functions: $h_{B_1}:=t_{B_1}^{\al{A}}=x^{\al{A}}$ (the identity function on
  $A=A^{(1)}$),
  $h_{C_1}:=t_{C_1}^{\al{A}}\restr A^{(2)}$, and
  $h_{D_1}:=t_{D_1}^{\al{A}}\restr A^{(3)}$.
  Namely, we have
  \begin{equation}
    \label{eq-fh}
    f^{\al{A}}(a,b,c)=\begin{cases}
    h_{B_1}(a)=a & \text{if $a=b=c$ or $a\not=b=c$,}\\
    h_{B_1}(c)=c & \text{if $a=b\not=c$,}\\
    h_{C_1}(a,b) & \text{if $a=c\not=b$,}\\
    h_{D_1}(a,b,c) & \text{if $a\not=b\not=c\not=a$}.
    \end{cases}
  \end{equation}
  This show that if $f^{\al{A}}$ obeys the identities in $\Sigma$, then there
  exist functions
  \mbox{$h_{B_1}\colon A\to A$}, $h_{C_1}\colon A^{(2)}\to A$, and
  $h_{D_1}\colon A^{(3)}\to A$
  such that \eqref{eq-fh} holds and
  the functions $h_{B_1},h_{C_1},h_{D_1}$ satisfy the following conditions:
  \begin{enumerate}
  \item[(i)]
    $h_{B_1}$ is the identity function on $A=A^{(1)}$, and
  \item[(ii)]
    $h_{D_1}$ is invariant under permuting its first and third variables;
    or equivalently, $h_{D_1}$ is constant on the orbits of
    $G_{D_1}=\langle(x,z)\rangle$,
    as $G_{D_1}$ acts on $A^{(3)}$ by permuting coordinates.
  \end{enumerate}
  The last condition holds, because $f(z,y,x)\in D_1$ implies that
  the identity $t_{D_1}(x,y,z)=f(x,y,z)\approx f(z,y,x)=t_{D_1}(z,y,x)$ is
  entailed by $\Sigma$, so
  $t_{D_1}^{\al{A}}(a,b,c)=t_{D_1}^{\al{A}}(c,b,a)$ for all $(a,b,c)\in A^3$. 

  Conversely, we will now show that if we are given three functions
  $h_{B_1}\colon A\to A$, $h_{C_1}\colon A^{(2)}\to A^{(2)}$, and
  $h_{D_1}\colon A^{(3)}\to A^{(3)}$ satisfying conditions (i)--(ii) above, then
  the ternary operation $f^{\al{A}}$ defined by \eqref{eq-fh} obeys the
  identities in $\Sigma$. It is clear from the construction of $f^{\al{A}}$
  that it obeys the first identity, $f(x,x,y)\approx y$, in $\Sigma$.
  For the other identity, $f(x,y,z)\approx f(z,y,z)$, in $\Sigma$
  we can check
  \begin{equation}
    \label{eq-fsymm}
    f^{\al{A}}(a,b,c)=f^{\al{A}}(c,b,a)\quad \bigl((a,b,c)\in A^3\bigl)
  \end{equation}
  separately for each possible pattern of $(a,b,c)$.
  If $(a,b,c)\in A^{(3)}$, then the equality in \eqref{eq-fsymm}
  follows from the last
  line of the definition in \eqref{eq-fh} and property (ii) of $h_{D_1}$.
  If $a=c$, then the equality in \eqref{eq-fsymm} is trivial.
  Finally, if $a=b$ or $b=c$ (including the possibility that $a=b=c$),
  then the equality in \eqref{eq-fsymm} follows from the first two
  lines of the definition in \eqref{eq-fh}.

  This proves that there is a one-to-one correspondence
  between the models of $\al{A}$ of $\MM$ with universe $A$
  and the triples of functions
  $h_{B_1}\colon A\to A$, $h_{C_1}\colon A^{(2)}\to A$, and
  $h_{D_1}\colon A^{(3)}\to A$
  satisfying conditions (i)--(ii) above.
  Given such a triple of functions, the operation $f^{\al{A}}$ of $\al{A}$
  is obtained by formula \eqref{eq-fh}.
\qeddiamond
\end{exmp}

Now we give a precise description of the construction
of all models for any
$\MM=(\LL,\Sigma)$
satisfying Global Assumption~\ref{global-assumption}.
As in Example~\ref{exmp:cmaltsev2}, we will work with a maximal family
of essentially different linear terms for $\MM$,
where the terms have essential variables only.
Throughout this discussion we 
will use the notation and the facts established in
Theorem~\ref{thm-kelly},
Lemmas~\ref{lm-ess-vars}--\ref{lm-symm}, and
Definition~\ref{df-ess-vars_symm} without further reference.

Let $X=\{x_1,\dots,x_m\}$ be a large enough set of variables for $\MM$
where $x_1,\dots,x_m$ are all distinct, and the subscripts $1,\dots,m$
fix an ordering of these variables. 
Furthermore, let
$\mathcal C$ be a transversal for the $S_X$-orbits of equivalence
classes of
$\EE:=\EE_X^{\MM}$ such that for each $C\in \mathcal C$,
the set of essential variables of the terms in $C$ is
$X_C=\{x_1,\dots,x_{m_C}\}$.
Clearly, such a transversal exists, and since $\MM$ is idempotent,
there is a unique $C\in\mathcal{C}$
with $m_C=1$, namely the $\EE$-class containing the term $x_1$.
Now choose for every $C\in\mathcal{C}$ a term
$t_C=t_C(x_1,\dots,x_{m_C})$ in $C$ so that $t_C$ 
includes precisely
the variables in $X_C$
(i.e., all variables of $t_C$ are essential).
Moreover, assume that for the unique $C\in\mathcal{C}$ with $m_C=1$
we choose $t_C=x_1$.

We wish to argue that every
model $\al{A}$ of $\MM$
is determined by the
following indexed family of functions:
\begin{equation}
  \label{eq-term_mfam}
  (t_C^{\al A}\restr{A^{(m_C)}}:C\in \mathcal{C}).
\end{equation}  
Note that the indexed family~\eqref{eq-term_mfam}
depends on the choice of $\mathcal{C}$,
however, it does not depend on the choice of the terms
$t_C(x_1,\dots,x_{m_C})\in C$, because for any other term
$t'_C(x_1,\dots,x_{m_C})\in C$ (whose variables are all essential)
we have that
$\Sigma\models t_C(x_1,\dots,x_{m_C})\approx t'_C(x_1,\dots,x_{m_C})$.

Let $\al{A}$ be a model of $\MM$. To show that $\al{A}$ is determined by
the family~\eqref{eq-term_mfam},  
let $f\in \LL$ be a $d$-ary operation symbol. A
complete specification of $f^{\al A}$ can be obtained by defining it
separately on each
set $A^{(\mu)}$ as $\mu $ ranges over all patterns on $[d]$.
Let $\mu $ be a
pattern and choose any
$d$-tuple $(z_1,\dots,z_d)$ of variables in $X$ 
such that
$(z_1,\dots,z_d)$ has pattern $\mu $.
The linear term $f(z_1,\dots,z_d)$
lies in some equivalence class of~$\EE$,
and hence it lies in the $S_X$-orbit
of exactly one $C\in\mathcal{C}$. 
Thus, $f\bigl(\gamma(z_1),\dots,\gamma(z_d)\bigr) \in C$
for some $\gamma\in S_X$.
Since $(z_1,\dots,z_d)$ and $\bigl(\gamma(z_1),\dots,\gamma(z_d)\bigr)$
have the same pattern, we may assume without loss of generality that
$(z_1,\dots,z_d)$ was chosen so that
$f(z_1,\dots,z_d) \in C\in \mathcal C$.
Thus
\begin{equation*}
  \Sigma \vDash
f(z_1,\dots,z_d)
\approx t_C(x_1,\dots,x_{m_C}).
\end{equation*}
Note that by the definition of $X_C$ each of the $x_i$'s must appear
in the list
$(z_1,\dots,z_d)$.
Thus there is a function
$\sigma\colon [m_C] \to [d]$
with
$x_i=z_{\sigma(i)}$ for all $i\in [m_C]$. 
Since $\al{A}$ is a model of $\Sigma$,
it must be the case that for every $(a_1,\dots,a_d) \in A^{(\mu)}$ we have
that $f^{\al A}(a_1,\dots,a_d) = t_{C}^{\al A}(a_{\sigma(1)},\dots,a_{\sigma(m_C)})$. 

\begin{df}
  \label{df-mfam}
For a fixed $\mathcal{C}$ and for any
nonempty set $A$, let
us call an indexed family
$(h_C : C\in \mathcal C)$
of functions an \emph{$\MM$-family on $A$}
(suppressing reference to the choice of $\mathcal{C}$, for simplicity)
if,  for each $C\in \mathcal C$
\begin{enumerate}
\item
  $h_C$ is a function $A^{(m_C)} \to A$;
\item
$h_C$ is invariant under all permutations $\pi\in G_C$ of its
variables $x_1,\dots,x_{m_C}$;
equivalently,
$h_C$ is constant on the orbits of the symmetry group $G_C=G_{t_C}$
of $t_C$ as $G_C$ acts on $A^{(m_C)}$
by permuting coordinates;
\item if $m_C=1$ then
  $h_C(a_1)=a_1$ for all $a_1\in A^{(1)}=A$.
\end{enumerate}
\end{df}

The discussion preceding Definition~\ref{df-mfam} and the idempotence of $\MM$
imply that, if
$\al{A}$ is a model of $\MM$
then the indexed family $(h_C: C\in\mathcal{C})$ of functions defined by
$h_C:=t^{\al A}_C\restr{A^{(m_C)}}$ for every $C\in\mathcal{C}$ 
is an $\MM$-family
on $A$.
Moreover, the operations of $\al{A}$ can be obtained from this $\MM$-family
as follows:

\addtocounter{equation}{1}

\begin{equation*}
  \tag{\theequation}\label{eq-fs_from_hs}
\text{%
\begin{minipage}[t]{.85\hsize}
  For every $f\in \LL$ with
  arity $d$, for
  every pattern $\mu$ on $[d]$, and for
  every $(a_1,\dots,a_d) \in A^{(\mu)}$ we have
  \begin{equation*}
    f^{\al A}(a_1,\dots,a_d) =
    h_C(a_{\sigma(1)},\dots,a_{\sigma(m_C)})
  \end{equation*}
  where $C$ is the unique member of $\mathcal{C}$ such that
  $f(z_1,\dots,z_d)\in C$ for some tuple $(z_1,\dots,z_d)$
  of variables with pattern $\mu$, and
  $t_C(z_{\sigma(1)},\dots,z_{\sigma(m_C)})$ is the chosen representative
  of $C$ with $X_C=\{x_1,\dots,x_{m_C}\}$ and $x_i=z_{\sigma(i)}$ for all
  $i\in [m_C]$.
\end{minipage}
}
\end{equation*}  

\medskip
\noindent
To summarize, every model $\al{A}$ of $\MM$ is determined by the $\MM$-family
  \eqref{eq-term_mfam}
associated to $\al{A}$. 

We now consider the converse: every \MM-family induces an algebraic
structure on its underlying set. Moreover, that algebra will be a model of \MM. This is
the content of the following theorem.

\begin{thm}
  \label{thm-alg-descr}
  Let $X=\{x_1,\dots,x_m\}$
  be a set of $m$ distinct variables
  that is large enough for~$\MM$, and
choose a transversal $\mathcal{C}$ for the $S_X$-orbits of equivalence classes
of $\EE=\EE^\MM_X$ such that for each $C\in\mathcal{C}$ the set of
essential variables is $X_C=\{x_1,\dots,x_{m_C}\}$.
Furthermore, for each $C\in\mathcal{C}$, choose a term
$t_C=t_C(x_1,\dots,x_{m_C})$ in $C$ with essential variables only.
Then the following hold for every nonempty set $A$.
\begin{enumerate}
\item[{\rm(1)}]
The mapping
\begin{equation}
  \label{eq-model_mfam}
\al{A}\mapsto(t^{\al A}_C\restr{A^{(m_C)}}:C\in\mathcal{C})
\end{equation}
is a one-to-one correspondence between the
models of $\MM$ with universe $A$
and the $\MM$-families
of functions on $A$.
\item[{\rm(2)}]
  For any $\MM$-family $(h_C:C\in\mathcal{C})$ on $A$,
  the member functions $h_C$ with $m_C>1$
    can be chosen independently.
\item[{\rm(3)}]
  If $h_C\colon A^{(m_C)}\to A$ is a member function 
  of an $\MM$-family $(h_C:C\in\mathcal{C})$ on $A$
  such that $m_C>1$, then
  \begin{enumerate}
  \item[{\rm(i)}]
    $h_C$ is a disjoint union
     \[
     h_C=\bigcup\bigl(h_C\restr{D^{(m_C)}}:D\subseteq A,\ |D|=m_C\bigr)
     \]
     of its restrictions $h_C\restr{D^{(m_C)}}\colon D^{(m_C)}\to A$
     to the subsets $D^{(m_C)}$ of $A^{(m_C)}$ where $|D|=m_C$;
  \item[{\rm(ii)}]
    these restrictions $h_C\restr{D^{(m_C)}}$
    can be chosen independently; and
  \item[{\rm(iii)}]
    each such restriction $h_C\restr{D^{(m_C)}}$ is constant on the
    orbits of the symmetry group
    $G_C=G_{t_C}$ of $t_C$ (as $G_{C}$
    acts on $D^{(m_C)}$
    by permuting coordinates), and is otherwise arbitrary.
  \end{enumerate}  
  \end{enumerate}  
\end{thm}  

\begin{proof}
We start with the proof of statement (1).
In our discussion that led up to the statement of
Theorem~\ref{thm-alg-descr} we proved that \eqref{eq-model_mfam}
is a one-to-one mapping that assigns an $\MM$-family on $A$
to each model of $\MM$ with universe $A$.
It remains to show that this mapping  is onto.

So, let $(h_C: C\in \mathcal C)$ be an $\MM$-family on $A$, and for each $f\in\LL$ with
arity $d$ define an operation $f^{\al A}$ on $A$ as described in $(\dagger)$.  Then
$f^{\al{A}}$ is defined on the whole set $A^d$, because for every pattern $\mu$ on $[d]$,
the required objects $C$, $f(z_1,\dots,z_d)$, $\sigma$, and
$t_C(z_{\sigma(1)},\dots,z_{\sigma(m_C)})$ exist.  To see that $f^{\al{A}}$ is
well-defined, note that $C$ is uniquely determined by
$f$ and the pattern $\mu$ on $[d]$, but
there might be multiple choices for $f(z_1,\dots,z_d)$. Therefore we need to show that the
definition of $f^{\al{A}}\restr A^{(\mu)}$ does not depend on the choice of
$f(z_1,\dots,z_d)$.  To this end, the following claim will be useful.

\begin{clm}
\label{clm-welldef}
Let $f\in\LL$ be an operation symbol with arity $d$,
let $\mu$ be a pattern on $[d]$, and let
$C$ be the unique member of $\mathcal{C}$ such that
$f(z_1,\dots,z_d)\in C$ for some tuple $(z_1,\dots,z_d)$
of variables with pattern $\mu$.
Furthermore, let $\sigma$ be a function $[m_C]\to[d]$
such that $x_i=z_{\sigma(i)}$ for all $i\in [m_C]$.
If $(w_1,\dots,w_d)$ is another tuple of variables
with pattern $\mu$ such that 
$f(w_1,\dots,w_d)\in C$, and
$\tau$ is a function $[m_C]\to[d]$
such that $x_i=w_{\tau(i)}$ for all $i\in [m_C]$, then
\begin{enumerate}
\item
  $\{z_{\tau(i)}:i\in[m_C]\}=X_C$,
\item
  the map $\pi\colon X_C\to X_C$ defined by $x_i=z_{\sigma(i)}\mapsto z_{\tau(i)}$
  for every $i\in [m_C]$ is a permutation of $X_C$, and
\item
  $\pi\in G_C$; or equivalently,
  \begin{equation}
  \label{eq-tc}
  \Sigma\models t_C(x_1,\dots,x_{m_C})=t_C(z_{\sigma(1)},\dots,z_{\sigma(m_C)})\approx
  t_C(z_{\tau(1)},\dots,z_{\tau(m_C)}).
  \end{equation}
\end{enumerate}
\end{clm}  

\begin{proof}[Proof of Claim~\ref{clm-welldef}]
Let $f$, $\mu$, $C$, $(z_1,\dots,z_d)$, $(w_1,\dots,w_d)$,
$\sigma$, and $\tau$ satisfy the assumptions of the claim.
The terms  
$f(z_1,\dots,z_d)$, $f(w_1,\dots,w_d)$, and $t_C(x_1,\dots,x_{m_C})$
all belong to $C$, therefore
\begin{align}
  \Sigma\models{}
  & f(z_1,\dots,z_d)\approx t_C(z_{\sigma(1)},\dots,z_{\sigma(m_C)})
                      \quad\text{and}\label{eq-one}\\
  \Sigma\models{}
  & f(w_1,\dots,w_d)\approx t_C(w_{\tau(1)},\dots,w_{\tau(m_C)}).\label{eq-two}
\end{align}
Since $(z_1,\dots,z_d)$ and $(w_1,\dots,w_d)$ have the same pattern,
the map $\{w_1,\dots,w_d\}\to\{z_1,\dots,z_d\}$ defined by $w_j\mapsto z_j$
for all $j\in [d]$ is a bijection.
Hence, by changing variables in the
identity in \eqref{eq-two}, we get that
\begin{equation}
  \label{eq-three}
  \Sigma\models f(z_1,\dots,z_d)\approx t_C(z_{\tau(1)},\dots,z_{\tau(m_C)}).
\end{equation}
Now \eqref{eq-one} and \eqref{eq-three} imply that
\eqref{eq-tc} holds, where the equality follows from the fact that
$x_i=z_{\sigma(i)}$ for all $i\in [m_C]$.
In particular, \eqref{eq-tc} yields that
$t_C(z_{\tau(1)},\dots,z_{\tau(m_C)})\in C$.
Moreover, since $X_C=\{x_1,\dots,x_{m_C}\}$
is the set of essential variables of every term in $C$, we also get that
$\{z_{\tau(1)},\dots,z_{\tau(m_C)}\}=X_C$. Consequently, the map
$x_i=z_{\sigma(i)}\mapsto z_{\tau(i)}$ ($i\in [m_C]$) is a permutation of $X_C$.
These considerations prove statements (1) and (2) of the claim,
and also the displayed
line \eqref{eq-tc}
in statement (3). The fact that \eqref{eq-tc} 
is equivalent to $\pi\in G_C$ follows from the last statement in
Lemma~\ref{lm-symm}.
\qedsymbdiamond
\end{proof}

We return to proving that the operations of $\al{A}$ are well-defined,
that is, if $f\in\LL$ has arity $d$, $\mu$ is a pattern on
$[d]$, and $C$ is the unique member of $\mathcal{C}$ which contains
$f(z_1,\dots,z_d)$ for some tuple $(z_1,\dots,z_d)$ of variables with pattern
$\mu$, then the definition of $f^{\al{A}}\restr A^{(\mu)}$ by $(\dagger)$
is independent of
the choice of the term $f(z_1,\dots,z_d)$ (and the function $\sigma$).
So, let $(z_1,\dots,z_d)$ and $(w_1,\dots,w_d)$ be two such tuples
of variables, and let $\sigma$ and $\tau$ be the associated functions
required in $(\dagger)$.
Thus, $f$, $\mu$, $C$, $(z_1,\dots,z_d)$, $(w_1,\dots,w_d)$,
$\sigma$, and $\tau$ satisfy the assumptions of Claim~\ref{clm-welldef}.
If we define $f^{\al{A}}\restr A^{(\mu)}$ using the term
$f(z_1,\dots,z_d)$, we get that
\begin{equation}
  \label{eq-firstdef}
f^{\al{A}}(a_1,\dots,a_d):=h_C(a_{\sigma(1)},\dots,a_{\sigma(m_C)})
\quad \text{for all $(a_1,\dots,a_d)\in A^{(\mu)}$},
\end{equation}
while if we define $f^{\al{A}}\restr A^{(\mu)}$ using the term
$f(w_1,\dots,w_d)$, we get that
\begin{equation}
  \label{eq-seconddef}
f^{\al{A}}(a_1,\dots,a_d):=h_C(a_{\tau(1)},\dots,a_{\tau(m_C)})
\quad \text{for all $(a_1,\dots,a_d)\in A^{(\mu)}$}.
\end{equation}
By Claim~\ref{clm-welldef}, the assignment defined by
$x_i=z_{\sigma(i)}\mapsto z_{\tau(i)}$ for all $i\in [m_C]$
defines a permutation $\pi$ of the set $X_C=\{x_1,\dots,x_{m_C}\}$
of variables of $h_C$, and $\pi\in G_C$.
Therefore,
the tuples $(a_{\sigma(1)},\dots,a_{\sigma(m_C)}),\,
(a_{\tau(1)},\dots,a_{\tau(m_C)})\in A^{(m_C)}$
are in the same orbit of $G_C$ (as $G_C$ acts by permuting coordinates).
Hence, by condition~(2) in the definition of an $\MM$-family
(Definition~\ref{df-mfam}), we get that
the right hand sides of the
equalities in \eqref{eq-firstdef} and \eqref{eq-seconddef} are equal.
This finishes the proof that the operations
of $\al{A}$ are well-defined.

Next we argue that the $\MM$-family \eqref{eq-term_mfam}
associated to $\al{A}$ coincides with $(h_C:C\in\mathcal{C})$.
Let $C\in\mathcal{C}$. Then
$t_C(x_1,\dots,x_{m_C})=f(x_{\phi(1)},\dots,x_{\phi(d)})$ for some
$f\in\LL$ of arity $d$ and some onto function
$\phi\colon[d]\to[m_C]$.
Let $\mu$ denote the kernel of $\phi$, and let
$\sigma\colon[m_C]\to[d]$ be any right inverse of $\phi$.
Hence, $(x_{\phi(1)},\dots,x_{\phi(d)})$ is a tuple of variables with pattern $\mu$
such that $f(x_{\phi(1)},\dots,x_{\phi(d)})\in C$.
Since for
every $m_C$-tuple $(a_1,\dots,a_{m_C})\in A^{(m_C)}$ the
$d$-tuple $(a_{\phi(1)},\dots,a_{\phi(d)})$ also has pattern $\mu$,
we get that
\begin{align*}
t_C^{\al{A}}(a_1,\dots,a_{m_C})
=f^{\al{A}}(a_{\phi(1)},\dots,a_{\phi(d)})
={} & h_C(a_{\phi(\sigma(1))},\dots,a_{\phi(\sigma(m_C))})\\
={} & h_C(a_1,\dots,a_{m_{C}}),
\end{align*}
where the second equality is a consequence of the definition of $f^{\al{A}}$.
This proves the desired equality $t_C^{\al{A}}\restr A^{(m_C)}=h_C$.

It remains to prove that $\al{A}$ is a model of $\MM$.
We will argue that $\al{A}$ satisfies every linear identity
$r\approx s$ (with variables in $X$) such that $\Sigma\models r\approx s$.
Let $\bar{\Sigma}$ denote the set of all these identities.
Hence, $r\approx s\in\bar{\Sigma}$ iff $(r,s)\in\EE$. 
Let $\bar{\Sigma}_{\mathcal{C}}$ denote the set of all identities
$r\approx s\in\bar{\Sigma}$ where $s=t_C(x_1,\dots,x_{m_C})$ for some
$C\in\mathcal{C}$.
For every positive integer $k\le m\,(=|X|)$, let
$\bar{\Sigma}(k)$ denote the set of all identities in $\bar{\Sigma}$
with variables in $\{x_1,\dots,x_k\}$, and let
$\bar{\Sigma}_{\mathcal{C}}(k):=\bar{\Sigma}_{\mathcal{C}}\cap\bar{\Sigma}(k)$.
Clearly,
\begin{align*}
&\bar{\Sigma}(1)
\subseteq\bar{\Sigma}(2)
\subseteq\dots
\subseteq\bar{\Sigma}(m-1)
\subseteq\bar{\Sigma}(m)
=\bar{\Sigma}\quad\text{and}\\
&\bar{\Sigma}_{\mathcal{C}}(1)
\subseteq\bar{\Sigma}_{\mathcal{C}}(2)
\subseteq\dots
\subseteq\bar{\Sigma}_{\mathcal{C}}(m-1)
\subseteq\bar{\Sigma}_{\mathcal{C}}(m)
=\bar{\Sigma}_{\mathcal{C}}.
\end{align*}

\begin{clm}
  \label{clm-id-reduction}
  The following conditions on $\al{A}$ are equivalent for every
  $k\in [m]$:
  \begin{enumerate}
  \item[{\rm(a)}]
    $\al{A}$ satisfies all identities in $\bar{\Sigma}(k)$;
  \item[{\rm(b)}]
    $\al{A}$ satisfies all identities in $\bar{\Sigma}_{\mathcal{C}}(k)$.
  \end{enumerate}
\end{clm}  

\begin{proof}[Proof of Claim~\ref{clm-id-reduction}]
The implication (a) $\Rightarrow$ (b) is obvious, since
$\bar{\Sigma}_{\mathcal{C}}(k)\subseteq\bar{\Sigma}(k)$.

To prove that (b) $\Rightarrow$ (a), assume (b) holds, and let
$r\approx s\in\bar{\Sigma}(k)$.
Then $r$ and $s$ belong to the same $\EE$-class, so
by renaming variables if necessary (i.e., by
replacing $r\approx s$ with $r[\gamma]\approx s[\gamma]$
for some permutation $\gamma\in S_X$)
we get an identity $r'\approx s'\in\bar{\Sigma}$ such that
$r',s'\in C$ for some $C\in\mathcal{C}$. Clearly,
$r\approx s$ holds in $\al{A}$ iff $r'\approx s'$ does.
Since $r',s'$, and
$t_C=t_C(x_1,\dots,x_{m_C})$ are in the same $\EE$-class $C$,
the identities $r'\approx t_C$ and $s'\approx t_C$ both
belong to $\bar{\Sigma}_{\mathcal{C}}$. We also have
that $x_1,\dots,x_{m_C}$ are essential variables of all terms in $C$,
including $r'$ and $s'$, which implies that $m_C\le k$.
Therefore the identities $r'\approx t_C$ and $s'\approx t_C$ both
belong to $\bar{\Sigma}_{\mathcal{C}}(k)$, and hence hold in $\al{A}$
by our assumption.
It follows that $r'\approx s'$ also holds in $\al{A}$, and
therefore so does $r\approx s$.
\qedsymbdiamond  
\end{proof}
  
By Claim~\ref{clm-id-reduction}, to show that $\al{A}$ satisfies all
identities in $\bar{\Sigma}=\bar{\Sigma}(m)$, it suffices to prove,
by induction on $k$, that
$\al{A}$ satisfies all identities in
$\bar{\Sigma}_{\mathcal{C}}(k)$ for $k=1,\dots,m$.
If $k=1$ and $r\approx t_C\in \bar{\Sigma}_{\mathcal{C}}(1)$,
then $m_C=1$, $t_C=x_1$, and
$r=f(x_1,\dots,x_1)$ for some $f\in\LL$ with arity $d$.
Hence, for every $a\in A$,
\[
t_C^{\al{A}}(a)=a
\quad\text{and}\quad
r^{\al{A}}(a)=f^{\al{A}}(a,\dots,a)=h_C(a)=a,
\]
where the last equality follows from property~(3) of $\MM$-families
(see Definition~\ref{df-mfam}).
This shows that the identity $r\approx t_C\in\bar{\Sigma}(1)$ holds in
$\al{A}$.

Now let $k>1$, let $r\approx t_C\in \bar{\Sigma}_{\mathcal{C}}(k)$,
and assume that $\al{A}$ satisfies
every identity in $\bar{\Sigma}_{\mathcal{C}}(k-1)$.
Hence, by Claim~\ref{clm-id-reduction},
$\al{A}$ satisfies
every identity in $\bar{\Sigma}(k-1)$ as well.
Our goal is to show that the identity $r\approx t_C$ holds in $\al{A}$.
By the induction hypothesis, there is nothing to prove if
$r\approx t_C\in \bar{\Sigma}_{\mathcal{C}}(k-1)$, so we will assume that
$r\approx t_C\in \bar{\Sigma}_{\mathcal{C}}(k)\setminus\bar{\Sigma}(k-1)$.
Then the variables occurring in $r$ are exactly $x_1,\dots,x_k$, and the
identity $r\approx t_C$ has the form
$r(x_1,\dots,x_k)\approx t_C(x_1,\dots,x_{m_C})$ with $m_C\le k$.
We need to show that for all $(a_1,\dots,a_k)\in A^k$,
\begin{equation}
  \label{eq-a_j}
  r^{\al{A}}(a_1,\dots,a_k)= t_C^{\al{A}}(a_1,\dots,a_{m_C}).
\end{equation}

If $(a_1,\ldots,a_k)\notin A^{(k)}$, that is, $a_1,\dots,a_k$ are
not all distinct, then there exists a function $\psi\colon [k]\to[k]$
such that $\psi$ is not onto and 
$(a_1,\dots,a_k)=(a_{\psi(1)},\dots,a_{\psi(k)})$; 
hence also $(a_1,\dots,a_{m_C})=(a_{\psi(1)},\dots,a_{\psi(m_C)})$.
The identity
\begin{equation}
  \label{eq-identity-psi}
r(x_{\psi(1)},\dots,x_{\psi(k)})\approx t_C(x_{\psi(1)},\dots,x_{\psi(m_C)}) 
\end{equation}
is obtained from $r\approx t_C\in\bar{\Sigma}$ by variable substitution,
so is also lies in $\bar{\Sigma}$.
Furthermore, \eqref{eq-identity-psi}
contains at most $|\psi([k])|\le k-1$ variables, therefore
it differs from an identity in $\bar{\Sigma}(k-1)$ by renaming
variables only.
Hence the induction hypothesis forces \eqref{eq-identity-psi} to hold in
$\al{A}$. Using this fact in the second equality below we conclude that
\begin{equation*}
  r^{\al{A}}(a_1,\dots,a_k)
  = r^{\al{A}}(a_{\psi(1)},\dots,a_{\psi(k)})
  = t_C^{\al{A}}(a_{\psi(1)},\dots,a_{\psi(m_C)})
  = t_C^{\al{A}}(a_1,\dots,a_{m_C}).
\end{equation*}
This proves \eqref{eq-a_j} in the case $(a_1,\dots,a_k)\notin A^{(k)}$.

Now let $(a_1,\dots,a_k)\in A^{(k)}$. 
The term $r=r(x_1,\dots,x_k)$ has the form $r=f(x_{\phi(1)},\dots,x_{\phi(d)})$
for some $f\in\LL$ with arity $d$ and some onto function
$\phi\colon [d]\to[k]$.
Recall that since $r\approx t_C\in\bar{\Sigma}$, we have that
$f(x_{\phi(1)},\dots,x_{\phi(d)})=r\in C$. 
Let $\mu$ denote the kernel of $\phi$, and let
$\sigma\colon [k]\to [d]$ be any right inverse of $\phi$.
Hence, $(x_{\phi(1)},\dots,x_{\phi(d)})$ is a tuple of variables with pattern $\mu$
such that $f(x_{\phi(1)},\dots,x_{\phi(d)})\in C$.
Since for the $k$-tuple $(a_1,\dots,a_k)\in A^{(k)}$ the
$d$-tuple $(a_{\phi(1)},\dots,a_{\phi(d)})$ also has pattern $\mu$,
we get that
\begin{align*}
r^{\al{A}}(a_1,\dots,a_k)
=f^{\al{A}}(a_{\phi(1)},\dots,a_{\phi(d)})
={} & h_C(a_{\phi(\sigma(1))},\dots,a_{\phi(\sigma(m_C))})\\
={} & h_C(a_1,\dots,a_{m_{C}}),
\end{align*}
where the second equality is a consequence of the definition of $f^{\al{A}}$.
We also have
\begin{equation*}
t_C^{\al{A}}(a_1,\dots,a_{m_C})=h_C(a_1,\dots,a_{m_C}),
\end{equation*}  
because we established earlier in this proof that
$h_C=t_C^{\al{A}}\restr A^{(m_C)}$.
Thus, \eqref{eq-a_j} holds for tuples $(a_1,\dots,a_k)\in A^{(k)}$ as well,
which
finishes the proof that $r\approx t_C\in\bar{\Sigma}(k)$ holds in $\al{A}$.
The proof of statement (1) is now complete.

Statements~(2)--(3) follow 
immediately from the definition of an $\MM$-family
(see Definition~\ref{df-mfam})
and from the fact that
for each $C\in\mathcal{C}$ the subsets $D^{(m_C)}$ of $A^{(m_C)}$
for different $m_C$-element subsets $D$ of $A$ are disjoint.
\end{proof}

\begin{cor}
  \label{cor-models_exist}
$\MM$ has models with universe $A$ for every nonempty set $A$.
\end{cor}  

\begin{proof}
  This follows from statement~(1) in Theorem~\ref{thm-alg-descr}, because
  $\MM$-families exist on every nonempty set.  
\end{proof}  

\section{Subalgebras of Random Models}
\label{sec-subalg}

As in the preceding sections, we will assume that $\MM=(\LL,\Sigma)$
satisfies our Global Assumption~\ref{global-assumption}.
Our aim is to prove that, with probability $1$,
a random finite model of $\MM$ has only `small' proper subalgebras, where
the meaning of `small' depends on some parameters of $\MM$.
To define these parameters we will use a maximal family of
essentially different linear terms for $\MM$
where every term has essential variables only --- just as we did in
the preceding section.
However, we will adopt a different notation,
which will be more convenient for our purposes here.

\begin{df}
  \label{df-pk}
    Let $x_1,\dots,x_m$ be distinct variables such that 
  $X=\{x_1,\dots,x_m\}$ is large enough for $\MM$,
  and let $\{t_i:0\le i\le r\}\subseteq \LT_X^\LL$
  be a maximal family of essentially
  different linear terms for $\MM$; that is, for $\EE=\EE_X^\MM$,
  the equivalence classes $\EE[t_i]$ ($0\le i\le r$) of the terms
  form a transversal for the orbits of $S_X$ on the set of all
  equivalence classes of $\EE$.
  Furthermore, assume that $t_0=x_1$, and for $i\in[r]$,
  we have $t_i=t_i(x_1,\dots,x_{d_i})$ where all variables $x_1,\dots,x_{d_i}$
  are essential, and
  \[  
d_\MM:=d_1=d_2=\dots=d_\ell<d_{l+1}\le \dots\le d_r\ (\ell\in[r]).  
  \]
  For arbitrary integer $k\ge d_\MM$ let
  \[
  p_\MM(k):=\sum_{i=1}^rq_i\binom{k}{d_i},
  \]
  where $q_i$ $(i\in [r])$
  is the index of the symmetry group $G_{t_i}$ of $t_i$ in the
  symmetric group $S_{\{x_1,\dots,x_{d_i}\}}$,
    and $\binom{k}{d_i}=0$ if $k< d_i$.
  \end{df}  

It is easy to see that in Definition~\ref{df-pk} we have
$d_\MM\ge2$, and $d_\MM$
is the minimum of the arities of minimal terms for $\MM$.
Moreover, it follows from the choice of the terms
$t_1,t_2,\dots,t_r$
that
the sequence $d_1\le d_2\le \dots\le d_r$ of their arities,
the set of all pairs $(d_i,q_i)$ ($i\in[r]$), and hence
the parameters $p_\MM(k)$ ($k\ge d_\MM$)
depend on $\MM$ only, they
are independent of the choice of the terms $t_1,t_2,\dots,t_r$.

Recall from Theorem~\ref{thm-alg-descr} that
to every model $\al{A}$ of $\MM$ there is an \emph{associated $\MM$-family}
$(h_i:0\le i\le r)$ which consists of the
functions $h_i=t_i^{\al{A}}\restr{A^{(d_i)}}$
($0\le i\le r$).
The models $\al{A}$ of $\MM$ on a fixed set $A$
can be reconstructed from the associated $\MM$-families
$(h_i:0\le i\le r)$;
moreover, $h_0$ is the identity function $A\to A$ and
the functions $h_1,\dots,h_r$
can be chosen independently so that
the following conditions hold for each $i\in[r]$:
\addtocounter{equation}{1}
\begin{equation*}
  \tag{\theequation}\label{eq-hs}
\text{%
  \begin{minipage}[t]{.87\hsize}
    \ \kern-15pt$\diamond$
    $h_i$ is a disjoint union
    of its restrictions $h_i\restr{D^{(d_i)}}\colon D^{(d_i)}\to A$
    with $|D|=d_i$,\\
    \phantom{i}\kern-15pt$\diamond$
    these restrictions can be chosen independently, and\\
    \phantom{i}\kern-15pt$\diamond$
    each such restriction $h_i\restr{D^{(d_i)}}$ is constant on the $q_i$
    orbits of the symmetry group $G_{t_i}$ of $t_i$ (as $G_{t_i}$
    acts on $D^{(d_i)}$
    by permuting coordinates), but is otherwise arbitrary.
\end{minipage}
}
\end{equation*}  

The significance of the parameters $p_\MM(k)$ ($k\ge d_\MM$)
is explained by the
following lemma.

\begin{lm}
  \label{lm-subalgB}
Let $A$ be a set, and $B$ a $k$-element subset of $A$ such that
$k\ge d_\MM$.

\begin{enumerate}
  \item[{\rm(1)}]
The following conditions are equivalent:
\begin{enumerate}
\item[{\rm(a)}]
  $B$ is (the universe of) a subalgebra of $\al{A}$;
\item[{\rm(b)}]
  the $\MM$-family $(h_i:0\le i\le r)$
  associated to $\al{A}$ has the property that
  \begin{equation} 
    \label{eq-subalgB}
    \text{$h_i(B^{(d_i)})\subseteq B$ for every $i\in [r]$.}
  \end{equation}
\end{enumerate}
\item[{\rm(2)}]
  If $|A|=n$, then
  in the probability space of random models $\al{A}$ of $\MM$ on $A$,
  \begin{equation*}
  \PrA\bigl(\text{$B$ is a subalgebra of $\al{A}$}\bigr)
  =\left(\frac{k}{n}\right)^{p_\MM(k)}.
  \end{equation*}
\end{enumerate}  
\end{lm}  

\begin{proof}
  We will use the notation of Definition~\ref{df-pk}.
  Let $\al{A}=\langle A;\LL\rangle$ be a random model of $\MM$, and let
  $(h_i:0\le i\le r)$ be the associated $\MM$-family.
  We will use Theorem~\ref{thm-alg-descr} in the form as it is restated in
  \eqref{eq-hs} and the paragraph preceding it.
  This theorem, combined with the description in 
  \eqref{eq-fs_from_hs} of how the operations of $\al{A}$ are constructed from
  the associated $\MM$-family, immediately imply that $B$ is a subalgebra of
  $\al{A}$ if and only if the requirement in condition~\eqref{eq-subalgB}
  holds for all $i$, $0\le i\le r$.
  Since $h_0$ is the identity function $A\to A$,
  it automatically satisfies this requirement, so there is no need
  for including the case $i=0$ in \eqref{eq-subalgB}.
  This proves statement~(1).

  To prove statement (2), we work in the probability space of all models of
  $\MM$ on $A$ where $|A|=n$.
  Combining statement~(1) above with the fact that the functions $h_i$
  ($i\in[r]$) are independent, we get that
  \begin{equation*}
  \PrA(\text{$B$ is a subalgebra})
  =\PrA\bigl(\text{$h_i(B^{(d_i)})\subseteq B$ for all $i\in[r]$}\bigr)
  =\prod_{i=1}^r\PrA\bigl(\text{$h_i(B^{(d_i)})\subseteq B$}\bigr).
  \end{equation*}

  Now, let us fix $i\in[r]$.
  By \eqref{eq-hs},
  the restrictions $h_i\restr{D^{(d_i)}}$ of $h_i$ to the different
  $d_i$-element subsets $D$ of $A$ are independent.
  Since $B^{(d_i)}$ is the union of all sets $D^{(d_i)}$ with $D\subseteq B$,
  $|D|=d_i$, the condition
  $h_i(B^{(d_i)})\subseteq B$ holds for $h_i$ if and only if
  $h_i(D^{(d_i)})\subseteq B$ for all $D\subseteq B$ with $|D|=d_i$.
  Thus,
  \[
  \PrA\bigl(\text{$h_i(B^{(d_i)})\subseteq B$}\bigr)
  =\prod_{\substack{D\subseteq B\\ |D|=d_i}}
        \PrA\bigl(\text{$h_i(D^{(d_i)})\subseteq B$}\bigr).
  \]
  Let $D\subseteq B$ with $|D|=d_i$.      
  Again by \eqref{eq-hs},
  the restriction $h_i\restr{D^{(d_i)}}$ of $h_i$ to $D$
  is constant on
  the $q_i$ orbits of the action of
  $G_{t_i}$ on the set $D^{(d_i)}$, and is otherwise
  arbitrary.
  Consequently, $h_i\restr{D^{(d_i)}}$ is determined by $q_i$
  free and independent choices of function values --- one for each orbit of
  $G_{t_i}$ on $D^{(d_i)}$. Thus, 
  \[
  \PrA\bigl(\text{$h_i(D^{(d_i)})\subseteq B$}\bigr)
  =\left(\frac{k}{n}\right)^{\!q_i}.
  \]
  By combining the results in the last three displayed lines we obtain that
  \[
  \PrA(\text{$B$ is a subalgebra})=
  \prod_{i=1}^r\prod_{\substack{D\subseteq B\\ |D|=d_i}}
  \PrA\bigl(\text{$h_i(D^{(d_i)})\subseteq B$}\bigr)
  =\prod_{i=1}^r
       \left(\frac{k}{n}\right)^{q_i\binom{k}{d_i}}
  =\left(\frac{k}{n}\right)^{p_\MM(k)}
  \]
  as claimed.
\end{proof}  

The following easy consequences of Lemma~\ref{lm-subalgB}(2) will be useful.

\begin{cor}
\label{cor-subalgB}
Let $A$ be an $n$-element set with $n>d_\MM$.
The following hold in the probability space of all models $\al{A}$ of $\MM$
on $A$.
\begin{enumerate}
\item[{\rm(1)}]
  For every integer $k$ such that $d_\MM\le k<n$,
  \begin{equation}
    \label{eq-subalgB2}
  \PrA\bigl(\text{$\al{A}$ has a $k$-element subalgebra}\bigr)
  \le\binom{n}{k}\left(\frac{k}{n}\right)^{p_\MM(k)}.
  \end{equation}  
\item[{\rm(2)}]
  For every integer $u\ge d_\MM$,
  \begin{equation}
    \label{eq-subalgB3}
    \PrA\bigl(\text{$\al{A}$ has a proper subalgebra of size $\ge u$}\bigr)
    \le\sum_{k=u}^{n-1}\binom{n}{k}\left(\frac{k}{n}\right)^{p_\MM(k)}.
  \end{equation}  
\end{enumerate}
\end{cor}

Our next lemma is an analog of a result
Murski\v\i\ used in \cite[pp. 50--51]{murskii} (see \cite[Lemma 6.22]{bergman}).
We postpone the proof to the Appendix.

\begin{lm}
  \label{lm-murskii}
  \begin{equation}
    \label{eq-murskii}
  \lim_{n\to\infty}\sum_{k=4}^{n-1}\binom{n}{k}\left(\frac{k}{n}\right)^{\binom{k}{2}}
  =0.
  \end{equation}
\end{lm}

Now we are ready to state and prove the main theorem of this section
on the sizes of
proper subalgebras of random finite models of $\MM$.

\begin{thm}
  \label{thm-subalg}  
  Let $\al{A}$ be a random finite model of $\MM$.
  \begin{enumerate}
  \item[{\rm(1)}]
    Every subset of $A$ of size less than $d_\MM$ is (the universe of)
    a subalgebra of $\al{A}$.
  \item[{\rm(2)}]
    The probability that $\al{A}$ has a proper subalgebra of size
    $d=d_\MM$ is
  \[
    {}\begin{cases}
      1           & \text{if\ \ $p_\MM(d)<d$,}\\
      1-e^{-d^d/d!} & \text{if\ \ $p_\MM(d)=d$,\ \ and}\\
      0           & \text{if\ \ $p_\MM(d)>d$.}
    \end{cases}
  \]
  \item[{\rm(3)}]
    The probability that $\al{A}$ has a $(d_\MM+1)$-element proper
    subalgebra is $0$ if \break $p_\MM(d_\MM+1)>d_\MM+1$,
    i.e., if one of the following conditions holds:
    \begin{itemize}
    \item
      $p_\MM(d_\MM)>1$, or
    \item
      there exists a linear term for $\MM$ with exactly $d_\MM+1$
      essential variables.
    \end{itemize}
  \item[{\rm(4)}]
    The probability that $\al{A}$ has a proper subalgebra of size
    $\ge d_\MM+2$ is $0$.
  \end{enumerate}    
\end{thm}  

Table~\ref{tab:subalg-size} summarizes most results of this theorem.
The four rows correspond to parts (1)--(4) of the theorem, and show
the probability for a random finite model of $\MM$ to have
a proper subalgebra of a specific size
$k=2,\dots,d_\MM-1$, or
$k=d_\MM$, or $k=d_\MM+1$, or of any size $k>d_\MM+1$,
respectively,
as a function of the parameter $p_\MM(d_\MM)$ of $\MM$.
(In the table $1^*$ indicates that an event is certain.)

As indicated by the ``?'' in Table~\ref{tab:subalg-size}, the
analysis in Theorem~\ref{thm-subalg} is incomplete:
we have been unable to establish a useful bound on the probability
that a random finite model of $\MM$ has a
$(d_\MM+1)$-element subalgebra,
provided the following two conditions hold for $\MM$:
\begin{enumerate}
\item[(a)]
  $p_\MM(d_\MM)=1$.\\
  Equivalently: there exists a unique $d_\MM$-ary minimal term $t$
  for $\MM$ and $t$ is totally symmetric (i.e., $t$ is invariant, modulo $\MM$,
  under all permutations of its $d_\MM$ variables), and hence
  by Theorem~\ref{thm-min-t}, $d_\MM\in\{2,3\}$ and
  $t$ is an essentially binary term or a minority term or a majority term for
  $\MM$.
\item[(b)]
  There is no linear term for $\MM$ with exactly $d_\MM+1$
  essential variables. 
\end{enumerate}
We leave this unsettled case as an open problem.

\begin{prb}\label{prb-1}
  Let $\MM$ satisfy our Global Assumption~\ref{global-assumption}.
If conditions (a)--(b) in the preceding paragraph hold for $\MM$,
what is the probability
that a random finite model of $\MM$ has a proper subalgebra of
size~$d_\MM+1$?
\end{prb}

\newlength{\tr} \tr=1.5pt
\newcolumntype{V}{!{\vrule width \tr }}
\newcommand{\thickh}[1]{\noalign{\global\arrayrulewidth\tr}\cline{#1}
  \noalign{\global\arrayrulewidth=.4pt}}

\begin{table}
  \centering
    \renewcommand{\arraystretch}{1.25}
   \begin{tabular}{|l|c|c|c|c|}
     \hline
     \hfill $p_\MM(d_\MM)$ & $=1$ & $=2, \dots, d_\MM-1$ & $=d_\MM$ &
             $=d_\MM+1,d_\MM+2,\dots$\\
     subalg size\phantom{$p_\MM(d_\MM)$} &&&&\\
     \hline
     $=2,\dots,d_\MM-1$ & $1^*$ & $1^*$ & $1^*$ & \multicolumn{1}{cV}{$1^*$}\\
     \cline{1-4}\thickh{5-5}
     $=d_\MM=d$& $1$ & 1& \multicolumn{1}{cV}{$1-e^{-d^d/d!}$} & 0\\
     \cline{1-2}\thickh{3-4}\cline{5-5}
     $=d_\MM+1$ & \multicolumn{1}{cV}{?} & 0 & 0 & 0\\
     \cline{1-1}\thickh{2-2}\cline{3-5}
     \multicolumn{1}{|lV}{$>d_\MM+1$} & 0 & 0 & 0 & 0\\
     \hline
   \end{tabular}
\medskip
  \caption{Probability of the existence of subalgebras of various sizes, as a function of $p_\MM(d_\MM)$}\label{tab:subalg-size}
\end{table}

The zeros in the lower half of Table~\ref{tab:subalg-size} witness that
with probability $1$, random finite models of $\MM$ have only `small'
proper subalgebras. We restate some of the results of Theorem~\ref{thm-subalg}
to emphasize this conclusion.

\begin{cor}
  \label{cor-subalg}
  If $\al{A}$ is a random finite model of $\MM$, then
  with probability $1$,
  \begin{itemize}
  \item
    $\al{A}$ has no proper subalgebras of any size $\ge d_\MM+2$;
  \item
    $\al{A}$ has no proper subalgebras of any size $\ge d_\MM+1$
    if either
    $p_\MM(d_\MM)>1$ or there exists a linear term for $\MM$ with
    exactly $d_\MM + 1$ essential variables; and 
  \item
    $\al{A}$ has no proper subalgebras of any size $\ge d_\MM$ if
    $p_\MM(d_\MM)>d_\MM$.
  \end{itemize}    
\end{cor}  

Now we prove Theorem~\ref{thm-subalg}.

\begin{proof}[Proof of Theorem~\ref{thm-subalg}]
  We will use the notation of Definition~\ref{df-pk}, but
  for simplicity, we will write $d$ for $d_\MM$ and
  $p(k)$ for $p_\MM(k)$ throughout the proof.
  Let $\al{A}$ be a random finite model of $\MM$, and let
  $(h_i:0\le i\le r)$ be the associated $\MM$-family. 
  Statement~(1) of the theorem
  is clearly true, because every linear term for $\MM$
  with less that $d$ variables is trivial
  (i.e., $\Sigma$-equivalent to a variable).
  For statements (2)--(4)
  we will assume without loss of generality that $|A|=n>d+2$.

  To prove statement (2), we first work in the probability space
  of all models $\al{A}$ of $\MM$ on a fixed set $A$. 
  Let $B$ be a $d$-element subset of $A$.
  We see from Lemma~\ref{lm-subalgB}(2) that
  the probability that $B$ is \emph{not} a subalgebra of $\al{A}$
  is $1-(\frac{d}{n})^{p(d)}$.
  Since for different
  $d$-element
  subsets $B$ and $C$ of $A$ the sets
  $B^{(d_i)}$ and $C^{(d_i)}$ are disjoint for all $i\in[r]$ (moreover,
  $B^{(d_i)}=C^{(d_i)}=\emptyset$ whenever $d_i>d$), we see that
  condition \eqref{eq-subalgB} and its version with $C$ replacing $B$
  are independent for any two different $d$-element subsets $B$ and $C$
  of $A$.
  Hence, it follows from Lemma~\ref{lm-subalgB}(1) that the events
  ``$B$ is a subalgebra of $\al{A}$'' and 
  ``$C$ is a subalgebra of $\al{A}$'' are also independent
  for any two different $d$-element subsets $B$ and $C$
  of $A$.
  Consequently,
  \[
  \PrA\bigl(\text{no $d$-element subset of $A$
    is a subalgebra of $\al{A}$}\bigr)
  =\left(1-\left(\frac{d}{n}\right)^{p(d)}\right)^{\binom{n}{d}}.
  \]
  
  This implies that for a random finite model $\al{A}$ of $\MM$,
  the probability that no $d$-element subset of $A$ is
  a subalgebra of $\al{A}$ is
  \[
  \lim_{n\to\infty}\left(1-\left(\frac{d}{n}\right)^{p(d)}\right)^{\binom{n}{d}}
  =\lim_{n\to\infty}\Bigl((1-x_n)^{x_n^{-1}}\Bigr)^{y_n}
  \]
  where $x_n:=\left(\frac{d}{n}\right)^{p(d)}$ and
  $y_n:=(\frac{d}{n})^{p(d)}\binom{n}{d}$.
  Since
  \[
  \lim_{n\to\infty}x_n=0,\ \ \lim_{n\to\infty}(1-x_n)^{x_n^{-1}}=e^{-1},\ \ 
  \text{and}\ \ 
  \lim_{n\to\infty}y_n
  =\begin{cases}
  \infty & \text{if $p(d)<d$,}\\
  d^d/d! & \text{if $p(d)=d$,}\\
  0      & \text{if $p(d)>d$,}\\
  \end{cases}
  \]
  we obtain that the probability that $\al{A}$ has \emph{no}
  $d$-element subalgebra
  is $0$, $e^{-d^d/d!}$, or $1$, according to whether $p(d)<d$, $p(d)=d$, or
  $p(d)>d$.
  Hence the probability that $\al{A}$ \emph{has}
  a $d$-element subalgebra
  is as claimed in (2).

  In statement (3),
\begin{align*}
  \label{eq-pk}
  p(d+1) & {}=\sum_{i=1}^\ell q_i\binom{d+1}{d}+
  \sum_{i=\ell+1}^r q_i\binom{d+1}{d_i}\\
  & ={}\left(\sum_{i=1}^\ell q_i\right)(d+1)+
  \sum_{\substack{i\ \text{with}\\ d_i=d+1}} q_i
  =p_{\MM}(d)(d+1)+\sum_{\substack{i\ \text{with}\\ d_i=d+1}} q_i.
\end{align*}  
Since $p(d)\ge1$, we have $p(d+1)\ge d+1$; moreover,
$p(d+1)>d+1$ holds if and only if $p(d)>1$ or $d_i=d+1$ for some $i\in[r]$.
This proves the correctness of the characterization of the condition
$p(d+1)>d+1$.

To verify the main statement of (3) on the probability of the existence of
$(d+1)$-element subalgebras, recall from Corollary~\ref{cor-subalgB}(1)
that for models $\al{A}$ of $\MM$ on a fixed $n$-element set with $n>d+2$,
the probability that $\al{A}$ has a $(d+1)$-element proper subalgebra
is bounded above by the quantity on the right hand side of
\eqref{eq-subalgB2} for $k=d+1$.
Therefore statement (3) will follow if we prove the following claim.

  \begin{clm}
    \label{clm-subalg1}
    If $p(k)> k$, then
    $\lim_{n\to\infty}\binom{n}{k}\left(\frac{k}{n}\right)^{p(k)}=0$.
  \end{clm}

  \begin{proof}[Proof of Claim~\ref{clm-subalg1}]
    \begin{equation*}
  \binom{n}{k}\left(\frac{k}{n}\right)^{p(k)}
  \le\binom{n}{k}\left(\frac{k}{n}\right)^{k+1}
  =\frac{k^{k+1}}{k!} \cdot \frac{n(n-1)\dots(n-k+1)}{n^k} \cdot \frac{1}{n}
  \le\frac{k^{k+1}}{k!} \cdot \frac{1}{n},
  \qquad
    \end{equation*}
    and for a fixed $k$, $\frac{k^{k+1}}{k!} \cdot \frac{1}{n}\to 0$
    as $n\to\infty$.
  \qedsymbdiamond  
  \end{proof}  

Finally, for the proof of statement (4), let $u=d+2$.
By Corollary~\ref{cor-subalgB}(2),
for models $\al{A}$ of $\MM$ on a fixed $n$-element set with $n>d+2$,
the probability that $\al{A}$ has a subalgebra of
size at least $u$ is bounded above by the sum on the
right hand side of \eqref{eq-subalgB3}.
Therefore
it suffices to show that this sum tends to $0$ as $n\to\infty$, as stated
in the following claim.

\begin{clm}
    \label{clm-subalg2}
    For $u=d+2$ we have that
    $\lim_{n\to\infty}\sum_{k=u}^{n-1}\binom{n}{k}\left(\frac{k}{n}\right)^{p(k)}=0$.
  \end{clm}

\begin{proof}[Proof of Claim~\ref{clm-subalg2}]
Notice that $u\ge4$, because $d\ge2$.
Furthermore, for all $k\ge u=d+2$ we have that
\[
p(k)\ge q_1\binom{k}{d_1}\ge\binom{k}{d_1}=\binom{k}{d}\ge\binom{k}{2},
\]
because $d\ge2$ and $k-d\ge u-d=2$.
Therefore,
  \[
  \sum_{k=u}^{n-1}\binom{n}{k}\left(\frac{k}{n}\right)^{p(k)}
  \le
  \sum_{k=u}^{n-1}\binom{n}{k}\left(\frac{k}{n}\right)^{\binom{k}{2}},
  \]
where the right hand side tends to $0$ as $n\to\infty$, by
  Lemma~\ref{lm-murskii}.
  \qedsymbdiamond  
  \end{proof}  
The proof of Theorem~\ref{thm-subalg} is complete.
\end{proof}

\section{Criterion for Random Models to be Almost Surely Idemprimal}
\label{sec-idempr}

As before, we will assume that $\MM$ satisfies our
Global Assumption~\ref{global-assumption}.
Our aim in this section is to show that under fairly mild additional
assumptions on
$\MM$, random models of $\MM$ are, with probability $1$, idemprimal.
Recall that an algebra $\al{A}$
is called \emph{idemprimal} if every idempotent operation on its universe is
a term operation of $\al{A}$. In particular, if $\al{A}=(A;\LL)$ is a model of
$\MM$, and hence is idempotent, it will be idemprimal if and only if the
term operations of $\al{A}$ are exactly the idempotent operations on $A$. 

Our main result is the following theorem.

\begin{thm}
  \label{thm-idempr}
  The following conditions on $\MM=(\LL,\Sigma)$ are equivalent:
  \begin{enumerate}
  \item[{\rm(a)}]
    With probability $1$, a random
    finite model of $\MM$ is idemprimal.
  \item[{\rm(b)}]
    The minimum arity $d_\MM$
    of a minimal term for $\MM$ is $2$ and
    $p_\MM(2)>2$.    
  \item[{\rm(c)}]
    There exist either
    \begin{itemize}
 \item
   three essentially different nontrivial binary terms
      for $\MM$,
      or
 \item
   two essentially different nontrivial binary terms, $s$ and $t$,
      for $\MM$
      such that $\Sigma\not\models s(x,y)\approx s(y,x)$.
 \end{itemize}  
  \end{enumerate}
\end{thm}

We will use a criterion for idemprimality,
which follows from \cite[Cor.~1.4]{szendrei}, and
characterizes idemprimal algebras --- among finite idempotent algebras ---
by forbidden compatible relations.
Recall that a $k$-ary relation $R\subseteq A^k$ on the universe $A$ of
an algebra $\al{A}$ is called a \emph{compatible relation of $\al{A}$}
if $R$ is the universe of a subalgebra of $\al{A}^k$.

\begin{thm}\cite{szendrei}
  \label{thm-szendrei}
  If $\al{A}$ is a finite idempotent algebra with universe
  $A$ of size $>2$,
  then $\al{A}$ is idemprimal if and only if it satisfies the following
  three conditions:
  \begin{itemize}
  \item
    $\al{A}$ has no proper subalgebras of size $>1$,
  \item
    $\al{A}$ has no nontrivial automorphisms (i.e., $\al A$ has
    no automorphisms other than the identity map), and
  \item
    no symmetric binary cross $\cross{a}:=(\{a\}\times A)\cup(A\times\{a\})$
    $(a\in A)$ is a compatible relation of $\al{A}$.
  \end{itemize}
\end{thm}  

Subalgebras of random models of $\MM$ were studied in
Section~\ref{sec-subalg}. In the forthcoming lemmas we will
discuss the probability
that a random model of $\MM$ has nontrivial automorphisms or
compatible crosses.
We start with a lemma which outlines a general method for proving that
these probabilities are $0$.
We will use the following notation:
\begin{align}
  &\text{$t=t(x_1,\dots,x_d)$
    is a minimal term of arity $d$
    for $\MM$, and}\label{eq-hq}\\
  &\text{$q$ is the index of $G_t$ in $S_{\{x_1,\dots,x_d\}}$;}\notag
\end{align}
for example, with the notation of Definition~\ref{df-pk}, we may choose
$t$ to be $t_1$, whence $d=d_1$ and $q=q_1$.
Furthermore, for every finite set $A$,
\begin{multline}
  \label{eq-Hn}
\qquad
\text{$H_A$ denotes the set of all
  functions $t^{\al{A}}\restr A^{(d)}$}\\
\text{as $\al{A}$ runs over all models
  $\al{A}=\langle A;\LL\rangle$ of $\MM$.}
\qquad
\end{multline}
By \eqref{eq-hs}, we have that
\addtocounter{equation}{1}
\begin{equation*}
  \tag{\theequation}\label{eq-H_n_descr}
\text{%
  \begin{minipage}[t]{.87\hsize}
    $H_A$ is the set of all functions $h\colon A^{(d)}\to A$ such that\\    
    \phantom{mm}$\diamond$
    $h$ is the union
    of its restrictions $h\restr{D^{(d)}}$
    with $D\subseteq A$, $|D|=d$,\\
    \phantom{mm}$\diamond$
    these restrictions are independent, and\\
    \phantom{mm}$\diamond$
    each function $h\restr{D^{(d)}}$ is constant on the $q$
    orbits of the action\\
    \phantom{mm$\diamond$}
    of $G_t$ on $D^{(d)}$, and is otherwise arbitrary.
  \end{minipage}
}
\end{equation*}  

\noindent
For a $k$-ary relation $\rho$ on $A$ and $h\in H_A$, we will say
that \emph{$\rho$ is compatible with $h$} if for all
$r_1,\dots,r_d\in\rho$ such that the $d$-tuple $(r_{1i},\dots,r_{di})$
of $i$th coordinates of $r_1,\dots,r_d$ lies in $A^{(d)}$ for each $i$
($1\le i\le k$), we have that by applying $h$ coordinatewise to
$r_1,\dots,r_d$, we get a tuple $h(r_1,\dots,r_d)$ in $\rho$.
Clearly, if $h=t^{\al{A}}\restr A^{(d)}$ for an algebra $\al{A}$ such that $\rho$
is a compatible relation of $\al{A}$, then $\rho$ will be compatible with $h$.
We will use the following notation:
\begin{equation}
  \label{eq-H_n_rho}
  H_{A,\rho}=\{h\in H_A:\text{$\rho$ is compatible with $h$}\}.
\end{equation}

Now let $\RR$ be a function which assigns to every finite set $A$ a family
$\RR_A$ of (finitary) relations on $A$.
We will say that $(\RR_A:|A|<\omega)$ is a
\emph{homogeneous family of relations on finite sets} if
for every bijection $\tau\colon A\to B$ between finite sets $A,B$
we have that $\RR_B=\{\tau(\rho):\rho\in\RR_A\}$; that is, the
system is invariant under renaming elements of the base set.
In particular, each family $\RR_A$ is invariant under permuting elements
of the base set $A$.
Examples of homogeneous families of relations on finite sets
include the following:
\begin{enumerate}
\item[(i)]
  $\RR_A$ is the set of all equivalence relations on $A$;
\item[(ii)]
  $\RR_A$ is the set of all (graphs of) permutations on $A$;
\item[(iii)]
  $\RR_A$ is the set of all partial orders on $A$;
\item[(iv)]
  $\RR_A$ is the set of all crosses $\cross{a}$ ($a\in A$) on $A$.
\end{enumerate}
It is easy to see that $(\RR_A:|A|<\omega)$ is a homogeneous
family of relations on finite sets if and only if
\begin{itemize}
\item
  $\RR_{[n]}$ is a family of finitary relations on $[n]$ such that
  $\RR_{[n]}$ is invariant under all permutations of $[n]$, and
\item
  $\RR_A=\{\tau(\rho):\rho\in\RR_{[n]}\}$ whenever
  $|A|=n$ and $\tau\colon[n]\to A$ is a bijection.
\end{itemize}  
Our interest in homogeneous families of relations stems from the
following obvious fact:
if $(\RR_A:|A|<\omega)$ is a homogeneous
family of relations on finite sets, then the property
``$\al{A}$ has a compatible relation in $\RR_A$''
is an abstract property for finite algebras $\al{A}$;
therefore, for such families, it makes sense to ask what the probability
of this property is for finite models $\al{A}$ of $\MM$
(cf.~Definition~\ref{df-probab}).

\begin{lm}
  \label{lm-probab-zero}
Suppose we are given a homogeneous
family $(\RR_A:|A|<\omega)$ of finitary relations on finite sets
such that each $\RR_A$ is finite, and let $t$, $q$,
$H_A$, and $H_{A,\rho}$ $(\rho\in\RR_A)$
be as in \eqref{eq-hq}--\eqref{eq-H_n_rho}.
If
\begin{equation}
  \label{eq-probab-zero}
  \lim_{n\to0}\sum_{\rho\in\RR_{[n]}}
  \frac{|H_{[n],\rho}|}
       {|H_{[n]}|}=0,
\end{equation}       
then the probability that a random finite model $\al{A}$ of $\MM$ has a    
compatible relation in $\RR_A$ is $0$.
\end{lm}

\begin{proof}
First we will find
an upper bound for the probabilities of the events
``$\rho$ is a compatible relation of $\al{A}$'' ($\rho\in\RR_A$)
in the probability space of all finite
models of $\MM$ on a fixed $n$-element set $A$
with $n>d$.

By our discussion preceding the definition of $H_{A,\rho}$ in
\eqref{eq-H_n_rho},
if $\al{A}$ is a random model of $\MM$ on $A$
such that $\rho$ is a compatible relation of $\al{A}$,  
and $h=t^{\al{A}}\restr{A^{(d)}}$, then $h\in H_{A,\rho}$.
By the choice of $t$, this function $h$ is a member of the
$\MM$-family associated to $\al{A}$.
For this $\MM$-family
we will use Theorem~\ref{thm-alg-descr} in the form as it is restated in
\eqref{eq-hs} and the paragraph preceding it.
Since the members (of arity $>1$) of this $\MM$-family
are independent, and $h$ is one of them,
we get that
  \begin{equation}
    \label{eq-rho1}
  \PrA(\text{$\rho$ is a compatible relation of $\al{A}$})
  \le
  \PrA\bigl(h\in H_{A,\rho}\bigr)
  =
  \frac{|H_{A,\rho}|}{|H_A|}.
  \end{equation}
Thus,   
  \begin{multline}
    \label{eq-RR_ubound}
    \qquad
    \PrA(\text{$\al{A}$ has a compatible relation in $\RR_A$})\\
    \le\sum_{\rho\in\RR_A}
    \PrA(\text{$\rho$ is a compatible relation of $\al{A}$})
    \le\sum_{\rho\in\RR_A} \frac{|H_{A,\rho}|}{|H_A|}.
    \qquad
  \end{multline}
The probability that a finite model $\al{A}=\langle A;\LL\rangle$
  of $\MM$ has a compatible relation
  in $\RR_A$ is the limit, as $n\to\infty$, of the probability
  estimated in \eqref{eq-RR_ubound} for $A=[n]$.  
  As a consequence of
assumption~\eqref{eq-probab-zero},
this limit is $0$, which completes
  the proof of Lemma~\ref{lm-probab-zero}.
\end{proof}

\begin{lm}
  \label{lm-aut}
  If $\al{A}$ is a random finite model of $\MM$, then
  the probability that $\al{A}$ has a nontrivial
  automorphism is $0$.
\end{lm}

\begin{proof}
  For every finite set $A$ let $\RR_A$ denote the set of all binary relations
  on $A$ which are graphs of nonidentity permutations of $A$.
  It is easy to see that
  $\rho\in\RR_A$ is a compatible relation of an algebra $\al{A}$
  on $A$ if and only if $\rho$ (considered as a function $A\to A$) is an
  automorphism of $\al{A}$. Clearly,
  $(\RR_A:|A|<\omega)$ is a homogeneous family of relations
  on finite sets.
  Therefore, our statement will follow from Lemma~\ref{lm-probab-zero}
  if we prove that \eqref{eq-probab-zero} holds for this choice of $\RR_A$.

  Let $\al{A}$ be a random model of $\MM$ on an $n$-element set $A$,
  let $\rho\in\RR_A$, and select $a,b\in A$ such that
  $\rho(a)=b\not=a$.
  It is easy to see that for any function $h\in H_A$, $\rho$ is compatible with
  $h$ --- i.e., $h\in H_{A,\rho}$ --- if and only if
\begin{equation}
  \label{eq-aut}
  h\bigl(\rho(a_1),\dots,\rho(a_d)\bigr)=\rho\bigl(h(a_1,\dots,a_d)\bigr)
  \quad 
  \text{for all}\ \ (a_1,\dots,a_d)\in A^{(d)}.
  \end{equation}
  
  Now let $\Gamma_\rho$ denote the family of all $d$-element subsets
  $C$ of $A$ such that $a\in C$ and $b\notin C$, and let
  $\Delta_\rho:=\{\rho(C):C\in\Gamma_\rho\}$.
  Then the assumption
  $\rho(a)=b\not=a$ implies that the sets in $\Delta_\rho$ contain $b$,
  while the sets in $\Gamma_\rho$ don't.
  Hence, $\Gamma_\rho\cap\Delta_\rho=\emptyset$, and therefore
\begin{equation}
  \label{eq-disjoint}
  \textstyle
  C^{(d)} \text{ is disjoint from }  \bigcup_{D\in\Delta_\rho}D^{(d)}
  \text{ for all } C\in\Gamma_{\rho}.
\end{equation}  
Furthermore, if $h\in H_{A,\rho}$, that is, if
  \eqref{eq-aut} holds for $h$, then
\begin{equation}
  \label{eq-determines}
  \text{$h\restr{C^{(d)}}$ with $C\in\Gamma_\rho$ determines
  $h\restr{D^{(d)}}$ with $D=\rho(C)\in\Delta_\rho$ via \eqref{eq-aut}.}  
\end{equation}

We claim that these facts imply the following inequalities:
\begin{equation}
  \label{eq-perm-ub}
  \frac{|H_{A,\rho}|}{|H_A|}\le\frac{1}{n^{q|\Delta_\rho|}}
  \le\frac{1}{n^{n-2}}.  
\end{equation}
For the proof we will use the description in \eqref{eq-H_n_descr} for
the functions $h\in H_A$.
  If $h\in H_{A,\rho}$, then by \eqref{eq-determines}, for each one of 
  the $|\Delta_\rho|$ choices of $D\in\Delta_\rho$, the constant values of
  $h$ on the $q$ orbits of $G_t$ on $D^{(d)}$ are determined by
  the function $h\restr{C^{(d)}}$ for some $C\in\Gamma_\rho$
  satisfying \eqref{eq-disjoint}.
  Hence, there is no choice for at least
  $q|\Delta_\rho|$ of the function values of $h$ that for an arbitrary member
  of $H_A$ could be chosen independently.
  Therefore
  $|H_{A,\rho}|\le |H_A|/n^{q|\Delta_\rho|}$, so
  the first inequality in \eqref{eq-perm-ub} follows.
  The second inequality in \eqref{eq-perm-ub}
  is a consequence of $q\ge1$ and
  $|\Delta_\rho|=|\Gamma_{\rho}|=\binom{n-2}{d-1}\ge n-2$
  (as $d\ge2$).

To prove \eqref{eq-probab-zero}
notice that
  $|\RRn|\le n!$. Hence, we get from \eqref{eq-perm-ub} (for $n>d$) that
  \[
  (0\le)\,
  \sum_{\rho\in\RRn}\frac{|H_{[n],\rho}|}{|H_A|}
    \le \frac{n!}{n^{n-2}}
  =n\left(\prod_{i=1}^{n-1}\sqrt{i(n-i)}\right)\frac{1}{n^{n-2}}
  \le n\Bigl(\frac{n}{2}\Bigr)^{n-1}\frac{1}{n^{n-2}}
  =\frac{n^2}{2^{n-1}},
  \]
  which implies that \eqref{eq-probab-zero} holds.
  This completes the proof of the lemma.
\end{proof}  

We now turn to
discussing the probability of
the presence of compatible crosses
in random finite models of $\MM$.
It is easy to see that if $\MM$ is the set of identities for
a single majority operation, then
in every model $\al{A}$ of $\MM$, all crosses $\cross a$ ($a\in A$)
are compatible relations of $\al{A}$.
Therefore the analog of Lemma~\ref{lm-aut} will not be true for
compatible crosses.
In this paper we will restrict to the case when $\MM$ has a minimal
binary term, which will be sufficient for the proof of
Theorem~\ref{thm-idempr}.

\begin{lm}
  \label{lm-cross}
  Assume $\MM$ has a minimal binary term.  
  If $\al{A}$ is a random finite model of $\MM$, then the
  probability that $\al{A}$ has a compatible symmetric cross 
  $\cross{a}$ $(a\in A)$ is $0$.
\end{lm}  

\begin{proof}
  Let $\al{A}$ be a random finite model of $\MM$ of size
$>2$, and let $t$ be a minimal binary term for $\MM$.
As before, we will use the notation \eqref{eq-hq}--\eqref{eq-H_n_rho}
with $d=2$, and will apply
Lemma~\ref{lm-probab-zero} to prove our claim. 
Let $\RR_A:=\{\cross{a}:a\in A\}$ for every finite set $A$.
As we mentioned earlier, $(\RR_A:|A|<\omega)$ is a homogeneous
family of relations on finite sets.
Therefore the statement of Lemma~\ref{lm-cross} will follow
if we prove that \eqref{eq-probab-zero} holds for this choice of $\RR_A$.

  Let $A$ be an $n$-element set ($n>2$),
  choose $a\in A$, and let $\rho=\cross{a}$.
  Further, let $\Delta_\rho$ denote the family of all $2$-element subsets
  $D$ of $A$ with $a\in D$.   
  We claim that for any $h\in H_A$, if $\rho$ is compatible
  with $h$ --- i.e., $h\in H_{A,\rho}$ --- then
  \begin{multline}
    \label{eq-cross3}
    \qquad
    \text{either\ \ $h(a,b)=a$ for all
      $D=\{a,b\}\in\Delta_\rho$},\\ 
    \text{or \ \ $h(b,a)=a$ for all
      $D=\{a,b\}\in\Delta_\rho$}.
    \qquad
  \end{multline}
  Indeed, otherwise there would exist $b,c\in A\setminus\{a\}$
  such that $h(a,b)\not=a$ and $h(c,a)\not=a$.
  This would imply $(a,c),(b,a)\in \rho$, $(a,b),(c,a)\in A^{(2)}$,
  and
  $h\bigl((a,c),(b,a)\bigr)=\bigl(h(a,b),h(c,a)\bigr)\notin\rho$,
  contradicting our assumption that $\rho$ is compatible with $h$. 

  Next we want to show that
  \begin{equation}
    \label{eq-cross-ub}
    \frac{|H_{A,\rho}|}{|H_A|}\le \frac{q}{n^{|\Delta_\rho|}}\le\frac{1}{n^{n-2}}.
    \end{equation}
  As in the proof of Lemma~\ref{lm-aut},
  we will use the description in \eqref{eq-H_n_descr} for the members of
  $H_A$ (with $d=2$).
  To estimate $|H_{A,\rho}|$, let $H_{A,\rho}'$, $H_{A,\rho}''$ 
  denote the sets of all
  $h\in H_{A,\rho}$ which satisfy the first option or the second
  option in \eqref{eq-cross3}, respectively.
  Clearly, $H_{A,\rho}=H_{A,\rho}'\cup H_{A,\rho}''$.
  If $q=1$ (i.e., $|G_t|=2$), then the functions $h=h(x,y)$ and $h(y,x)$
  coincide, and $H_{A,\rho}'=H_{A,\rho}=H_{A,\rho}''$.
  If $q=2$ (i.e., $|G_t|=1$), then $h=h(x,y)\mapsto h(y,x)$ yields a bijection
  $H_{A,\rho}'\to H_{A,\rho}''$.
  Therefore in both cases we have that
   \begin{equation}
    \label{eq-half}
    |H_{A,\rho}|\le q|H_{A,\rho}'|.
  \end{equation}
 
  If $h\in H_{A,\rho}'$, then for each one of 
  the $|\Delta_\rho|$ choices of $D=\{a,b\}\in\Delta_\rho$,
  the constant value of $h$ on the $G_t$-orbit of $(a,b)\in D^{(2)}$
  is uniquely determined (namely, it is $a$).
  Hence, there is no choice for at least
  $|\Delta_\rho|$ of the function values of $h$ that for an arbitrary member
  of $H_A$ could be chosen independently.
  This implies that
  $|H_{A,\rho}'|\le |H_A|/ n^{|\Delta_\rho|}$.
  Combining this inequality with \eqref{eq-half} we get 
  the first inequality in \eqref{eq-cross-ub}.
  The second inequality in \eqref{eq-cross-ub}
  follows from $q\le n$ and
  $|\Delta_\rho|=n-1$.

  Now \eqref{eq-probab-zero} is easy to prove.
  Since 
  $|\RRn|=n$, we get from \eqref{eq-cross-ub} (for $n>2$) that
  \[
  (0\le)\,
  \sum_{\rho\in\RRn}\frac{|H_{[n],\rho}|}{|H_A|}
    \le \frac{n}{n^{n-2}},
    \]
  which implies that \eqref{eq-probab-zero} holds.
  This completes the proof of the lemma.
 \end{proof}

We can now prove Theorem~\ref{thm-idempr} by combining the results
of this section with the results of the preceding section.

\begin{proof}[Proof of Theorem~\ref{thm-idempr}]
  The equivalence of conditions~(b) and (c) is an immediate consequence of
  the definitions of minimal terms and the parameter $p_\MM(2)$ for $\MM$.

To prove the equivalence of conditions (a) and (b),   
let $\al{A}$ be a random finite model of $\MM$.
First we show that (b)~$\Rightarrow$~(a).
Assume $\MM$ satisfies condition~(b). Then $p_\MM(2)>2$, therefore
we get from
Theorem~\ref{thm-subalg}(2)--(4) (or the last item of
Corollary~\ref{cor-subalg})
that, with probability $1$, $\al{A}$
has no proper subalgebras of size greater than $1$.
By Lemma~\ref{lm-aut}, we have that, with probability $1$, $\al{A}$ has
no nontrivial automorphisms.
Finally, by assumption~(b),
Lemma~\ref{lm-cross} applies, and
yields that, with probability $1$,
no symmetric cross $\cross{a}$ ($a\in A$) is a compatible relation of
$\al{A}$.
Thus, $\al{A}$ satisfies the idemprimality criterion in
Theorem~\ref{thm-szendrei} with probability $1$. Hence (a) follows.

Conversely, to prove (a)~$\Rightarrow$~(b), assume that (b) fails,
that is, either $\MM$ has no binary minimal term, or $1\le p_\MM(2)\le 2$.
In the former case every $2$-element subset of $A$ is (the universe of)
a subalgebra of $\al{A}$ by Theorem~\ref{thm-subalg}(1),
while in the latter case
$\al{A}$ has a $2$-element subalgebra with positive probability by
Theorem~\ref{thm-subalg}(2). Thus, (a) fails in all these cases.
\end{proof}  

\section{Application to Some Familiar Maltsev Conditions}
\label{sec-appl}

We conclude the paper
by considering our results in the context of a few well-known strong
idempotent linear Maltsev conditions.

\subsection{Hagemann--Mitschke terms for congruence $k$-permutable
  varieties.}
\label{ssec-kperm}
The language $\LL$ for Hagemann--Mitschke terms~\cite{hagMit}
consists of $k-1$ ternary operation symbols, $q_1, q_2,\dots,q_{k-1}$
($k\ge2$),
and the set $\Sigma$ of identities consists of 
\begin{align*}
  x&\approx q_1(x,y,y),\\
 q_i(x,x,y) &\approx q_{i+1}(x,y,y), \quad\text{for $i=1,2,\dots,k-2$},\\
  q_{k-1}(x,x,y) &\approx y.
\end{align*}
The Maltsev condition ``there exist Hagemann--Mitschke terms
$q_1, q_2,\dots,q_{k-1}$'' characterizes those varieties with $k$-permuting
congruences.
The system $\MM=(\LL,\Sigma)$ has $2k-3$ essentially different minimal terms,
all binary:
\[
q_1(x,y,x),\ q_2(x,y,x),\ \dots,\ q_{k-1}(x,y,x),\ q_1(x,x,y),\
q_2(x,x,y),\ \dots,\ q_{k-2}(x,x,y).
\]
Therefore, if $k\geq 3$ then there are at least $3$ such terms, so by
Theorem~\ref{thm-idempr}, random finite models of $\MM$ are almost surely
idemprimal.  

If $k=2$, then this conclusion fails.
In the case $k=2$ there is only one Hagemann--Mitschke term, $q_1$, which is
a Maltsev term.
We will discuss this case in the next subsection.

\subsection{Maltsev term for congruence ($2$-)permutable varieties.}
\label{ssec-perm}
The language $\LL$ for a Maltsev term~\cite{maltsev}
consists of a single ternary operation symbols, $f$,
and the set $\Sigma$ of identities consists of  
\begin{equation*}
  x\approx f(x,y,y)\qquad\text{and}\qquad f(x,x,y) \approx y.
\end{equation*}
Then $\MM=(\LL,\Sigma)$ has a single minimal term, namely
$f(x,y,x)$.
Since $\Sigma \nvDash f(x,y,x) \approx f(y,x,y)$, we obtain $p_{\MM}(2)=2$.
From Lemmas~\ref{lm-aut} and~\ref{lm-cross}, we learn that
random finite models of $\MM$ have, almost surely,
no nontrivial automorphisms or compatible crosses.
However, according to Theorem~\ref{thm-subalg}(2),
a random finite model of \MM will have a $2$-element
subalgebra with probability $1-e^{-2}$.
Thus by Theorem~\ref{thm-szendrei}, a random finite model
of $\MM$ will be idemprimal with probability $e^{-2}\approx 0.14$.

\subsection{Ternary minority term.}
A ternary minority term is a special kind of Maltsev term.
The existence of a ternary minority term is perhaps not of interest as a
Maltsev condition, but it does provide an interesting case study for
random models. The language describing a minority term is $\LL=\{f\}$,
and the set of identities is 
\[
\Sigma_1=\{f(x,y,y)\approx x,\ f(y,y,x) \approx x,\ f(y,x,y)\approx x\}.
\]
Clearly, $\MM_1= (\LL,\Sigma_1)$ has no binary minimal terms, 
so $d_{\MM_1}=3$ and the only minimal term for $\MM_1$ is $f(x,y,z)$
(up to renaming and permuting variables; cf.\ Theorem~\ref{thm-min-t}(2)).
The symmetry group of $f(x,y,z)$ is trivial, so
$p_{\MM_1}(3) = 6 > d_{\MM_1}$.
Theorem~\ref{thm-subalg}(1)--(4) tells us that every $2$-element subset of
every model of $\MM_1$ will be a subalgebra, but a random finite model
of $\MM_1$ will, almost surely, have no proper subalgebras
of size~$3$ or larger.
In particular, we see that no finite model of $\MM_1$ will be idemprimal.

Now set
\[
\Sigma_2= \Sigma_1\cup \{f(x,y,z) \approx f(y,z,x)\}.
\]
In $\MM_2=(\LL,\Sigma_2)$, the symmetry
group of $f(x,y,z)$ has order~$3$, thus $p_{\MM_2}(3)=2$. Then
Theorem~\ref{thm-subalg}(2)--(4)
implies that a random finite model of
$\MM_2$ almost surely has a $3$-element subalgebra, but no larger
proper subalgebra. 

Finally, with
\[
\Sigma_3=\Sigma_2\cup \{f(x,y,z)\approx f(y,x,z)\},
\]
the symmetry group of $f(x,y,z)$
is the full permutation group $S_{\{x,y,z\}}$.
At that point we find ourselves in the territory of
Problem~\ref{prb-1}, because $p_{\MM_3}(3)=1$
and there is no linear term for $\MM_3$ with exactly $4$ essential variables.
We know from Theorem~\ref{thm-subalg}(2) and~(4) that
a random finite model of
$\MM_3$ almost surely has a $3$-element subalgebra and no proper subalgebra
of size $5$ or larger, but the probability
that it has a $4$-element subalgebra is unclear.

\subsection{Near unanimity term.}
Let $k\geq 3$, and let $\LL=\{g\}$ where $g$ is a $k$-ary
operation symbol. The set of
\emph{$k$-ary near unanimity identities} is
\begin{equation*}
  \Sigma_k=\Bigl\{g(x,x,\dots,\overset{i^{\text{th}}}{y},x,\dots,x) \approx x :
  i=1,\dots,k\Bigr\}
\end{equation*}
where the lone $y$ appears in the $i^{\text{th}}$ position.  When $k=3$, $g$ is called a
\emph{majority term}.
In this case $g$ is the only minimal term for the system
$\MM_3=(\LL, \Sigma_3)$, so
by Theorem~\ref{thm-subalg}(1), every $2$-element subset of
every model of $\MM_3$ will be a subalgebra.
It follows that no finite model of $\MM_3$ will be idemprimal.

However, for $k>3$ there are $2^k-2k-2$ essentially different binary terms.
Thus by Theorem~\ref{thm-idempr}, a random finite model of
$\MM_k=(\LL,\Sigma_k)$ will almost surely be idemprimal.

\subsection{Further examples.}
In Tables~\ref{fig:idempr1}--\ref{fig:idempr2} we list the answers to the
question ``Is a random finite model of $\MM$ almost surely idemprimal?''
for many other systems $\MM$ which describe familiar strong idempotent
linear Maltsev conditions.
\begin{table}
  \centering
 \renewcommand{\arraystretch}{1.25}
  \begin{tabular}{|l|c|c|}
    \hline
    \multicolumn{1}{|c|}{System $\MM$} &
                  \multicolumn{2}{c|}{Is a random finite model of $\MM$}\\
    \multicolumn{1}{|c|}{for the Maltsev condition}
                  & \multicolumn{2}{c|}{almost surely idemprimal?}\\
    \cline{2-3}
    & YES, if & NO, if\\
    \hline \hline                 
    Hagemann--Mitschke terms $q_1,\dots,q_{k-1}$ \cite{hagMit}  &
                      $k\ge3$ & $k=2$\\                   
    \hfill for congruence $k$-permutability & & [Maltsev term]\\
    \hline    
    J\'onsson terms $t_0,\dots,t_k$ \cite{jonsson}  & $k\ge4$ & $k=2,3$\\  
    \hfill for congruence distributivity & & [$k=2$: $\sim$majority term]\\
    \hline
    Day terms $m_0,\dots,m_k$ \cite{day}  & $k\ge2$ & ---\\  
    \hfill for congruence modularity & &  \\
    \hline
    Gumm terms $d_0,\dots,d_k,p$ \cite{gumm} & $k\ge1$ & $k=0$\\
    \hfill for congruence modularity & & [$\sim$Maltsev term] \\ 
    \hline
    Terms $d_0,\dots,d_k$ \cite{hobbyMc,kearnesKiss} & $k\ge4$ & $k=2,3$\\
    \hfill for congruence join-semidistributivity &
          & [$k=2$: $\sim\frac{2}{3}$-minority term] \\ 
    \hline
    $k$-ary near unanimity term \cite{huhn,bakerPixley} & $k\ge4$ & $k=3$\\    
    & & [majority term]\\
    \hline
    $k$-cube term (arity: $2^k-1$) \cite{bimmvw} & $k\ge3$ & $k=2$\\
    & & [$\sim$Maltsev term] \\ 
    \hline
    $k$-edge term (arity: $k+1$) \cite{bimmvw} & $k\ge3$ & $k=2$\\
    & & [$\sim$Maltsev term] \\ 
    \hline
    $(m,n)$-parallelogram term & $m,n\ge1$ & ---\\
    \hfill(arity: $m+n+3$)\cite{parallelogram} & & \\
    \hline
    $k$-ary weak near unanimity term \cite{mckMaroti} & $k\ge4$ & $k=3$\\ 
    \hline
    $k$-ary cyclic term \cite{bkmmn} & $k\ge4$ & $k=2,3$\\
    \hline
  \end{tabular}
  \smallskip
  \caption{Idemprimality for random finite models of
    some familiar strong idempotent linear Maltsev conditions
    (with parameters)}
  \label{fig:idempr1}
\end{table}
\begin{table}
  \centering
 \renewcommand{\arraystretch}{1.25}
  \begin{tabular}{|l|c|}
    \hline
    \multicolumn{1}{|c|}{System $\MM$} &
                      Is a random finite model of $\MM$\\
    \multicolumn{1}{|c|}{for the Maltsev condition} &
                      almost surely idemprimal?\\
    \hline
    \hline                 
    $2/3$ minority term \cite{pixley2} & NO \\
    \hfill for characterizing arithmetical varieties & \\
    \hline
    Maltsev term and majority term \cite{pixley1} & NO \\
    \hfill for characterizing arithmetical varieties & \\
    \hline
    $6$-ary Siggers term \cite{siggers} & YES \\
    \hline
    $4$-ary Siggers term \cite{kmm} & YES \\
    \hline
    Ol\v{s}\'ak's $6$-ary weak $3$-cube term \cite{olsak} & YES \\
    \hline
  \end{tabular}
  \smallskip
  \caption{Idemprimality for random finite models of
    some familiar strong idempotent linear Maltsev conditions}
    \label{fig:idempr2}
\end{table}
Since for some of the Maltsev conditions in
Tables~\ref{fig:idempr1}--\ref{fig:idempr2}
the literature uses several slightly different descriptions $\MM$,
we followed the original papers, as indicated.

This introduces a minor inconsistency which has no effect on the random
models, as we now explain, using Hagemann--Mitschke terms as an example.
Instead of the system $\MM$ displayed in subsection~\ref{ssec-kperm},
Hagemann--Mitschke terms are often described by the system
$\MM'=(\LL',\Sigma')$ where $\LL'$ consists of the symbols $q_0,\dots,q_n$
and $\Sigma'$ consists of the identities
\begin{align*}
  x&\approx q_0(x,y,z),\\
 q_i(x,x,y) &\approx q_{i+1}(x,y,y), \quad\text{for $i=0,1,2,\dots,k-1$},\\
  q_k(x,y,z) &\approx z.
\end{align*}
The only difference between $\MM$ and $\MM'$ is that the second one has
two new symbols, $q_0$ and $q_k$, which are inessential, because
they are $\Sigma'$-equivalent to the variables $x$ and $z$, respectively,
and upon eliminating $q_0$ and $q_k$ by replacing them with those variables,
$\Sigma'$ becomes the same set of identities as $\Sigma$.
This implies that $\MM$ and $\MM'$ define equivalent varieties; in fact,
there is a one-to-one correspondence $\al{A}\mapsto\al{A}'$
between the models $\al{A}$ of $\MM$ and the models $\al{A}'$ of $\MM'$
such that the $\MM$-family of $\al{A}$ coincides with the $\MM'$-family
of $\al{A}'$. Therefore, the probability of every abstract property is the
same for the finite models of $\MM$ as for the finite models of $\MM'$.

In Table~\ref{fig:idempr1} the papers cited for
J\'onsson terms, Day terms, Gumm terms, and join semidistributivity terms
use the approach of including
inessential symbols in the language, namely $t_0,t_k$; $m_0,m_k$; $d_0$; and
$d_0,d_k$. 
Therefore, the claim that the
system for J\'onsson terms for $k=2$ reduces to the system
for a majority term $t_1$ is true only after eliminating the
inessential symbols $t_0,t_2$. This is indicated by the symbol $\sim$
in the last column of the table. The situation is similar for Gumm terms
and join semidistributivity terms.
For $k$-cube and $k$-edge terms, the case $k=2$ yields a Maltsev term
up to a possible permutation of variables only. This is what $\sim$ indicates
in the last column for those cases.

More interestingly, Table~\ref{fig:idempr1} also shows that it can happen that 
two different systems $\MM=(\LL,\Sigma)$ and
$\MM'=(\LL',\Sigma')$ satisfying our Global Assumption~\ref{global-assumption}
determine equivalent Maltsev conditions, but the question
``Are the random finite models almost surely idemprimal?''
for $\MM$ and $\MM'$ have different answers.
For example, for any fixed integer $k\ge2$,
let $\MM_k$ be the system
of identities for a $k$-cube term and let $\MM'_k$ be the system of identities
for a $(1,k-1)$-parallelogram term.
It was proved in \cite{parallelogram} that a variety has a $k$-cube term
if and only if it has a $(1,k-1)$-parallelogram term, so $\MM_k$ and $\MM'_k$
describe equivalent Maltsev conditions.
However, as Table~\ref{fig:idempr1} and the conclusion of 
subsection~\ref{ssec-perm} above indicate, in the case $k=2$
we have that a random
finite model of $\MM'_2$ is almost surely idemprimal, while   
a random finite model of $\MM_2$ is idemprimal only with probability
$e^{-2}$.

Now let us consider a pair of examples from Table~\ref{fig:idempr2}:
let $\MM$ be the system of identities for a $2/3$-minority term, and let
$\MM'$ be the system of identities in the language $\{f,d\}$ saying that
$f$ is a Maltsev term and $d$ is a majority term.
The corresponding Maltsev conditions are equivalent; both of them characterize
arithmetical varieties, by \cite{pixley1, pixley2}.
The answers to the question 
``Are the random finite models almost surely idemprimal?'' are also the same
for $\MM$ and $\MM'$. However, there are  
essential differences between the random finite models of $\MM$ and $\MM'$. 
By Theorem~\ref{thm-subalg}(1)--(2),
for the finite models of $\MM$ we have that
all $2$-element subsets are subalgebras, 
while for the finite models of $\MM'$,
there is a positive probability, namely $e^{-2}$, that a random
finite model has no $2$-element subalgebras.
The latter fact follows from our result in subsection~\ref{ssec-perm},
because the majority term $d$ makes no contribution to the family of
binary minimal terms for $\MM'$.

\subsection{Answers to our questions in the Introduction.}
In the first two paragraphs of the Introduction we mentioned several
questions involving specific Maltsev conditions, which have the following form
(in the interpretation discussed later on in the Introduction):
``For two given finite systems $\MM_1$ and $\MM_2$ of idempotent
linear identities (in a finite language), what is the probability that
a random finite model of $\MM_1$ satisfies the Maltsev condition
$\CC_{\MM_2}$ described by $\MM_2$?''

Our first question was this:
What is the probability that a random finite model of
the Hagemann--Mitschke identities for congruence $3$-permutability
(see subsection~\ref{ssec-kperm}) has a Maltsev term?
We discussed in subsection~\ref{ssec-kperm} that a random finite model of
the Hagemann--Mitschke identities for $k=3$ is almost surely idemprimal,
and therefore, by Corollary~\ref{cor-models_exist},
almost surely satisfies every strong idempotent linear Maltsev condition.
In particular, it follows that a random finite model of
the Hagemann--Mitschke identities for $k=3$ almost surely
has a Maltsev term.

We also asked: Will a random
finite algebra lying in a congruence
semidistributive variety almost surely generate one
that is congruence distributive? 
Since a variety is congruence semidistributive (i.e., both congruence
meet- and join-semidistributive) if and only if it is congruence
join-semidistributive, our interpretation of this question is the following: 
Will a random
finite model of the system $\SD_k(\vee)$ of identities
for join-semidistributivity terms for some $k\ge2$ almost surely
have J\'onsson terms for some $\ell\ge2$? (For the explicit identities, see the
papers cited in Table~\ref{fig:idempr1}.)
According to Table~\ref{fig:idempr1}, if $k\ge4$, then a random finite
model of $\SD_k(\vee)$ is almost surely idemprimal,
and therefore by the same argument as in the preceding paragraph,
it almost surely has J\'onsson terms for all $\ell\ge2$.
If $k=3$, then $\SD_3(\vee)$ and the system of identities for J\'onsson terms
for $\ell=3$ differ only by notation;
namely, one can be translated into the other by
the definition $t_i(x,y,z):=d_{3-i}(z,y,x)$~($0\le i\le 3$).
Finally, if $k=2$, then $\SD_k(\vee)$ implies that $d_1(x,y,z)$ is a
$\frac{2}{3}$-minority term. Hence,
$d_1(x,d_1(x,y,z),z)$ is a majority term, and therefore
J\'onsson terms for $\ell=2$ exist.
In summary, we see that for every $k\ge2$,
a random finite model of $\SD_k(\vee)$ will almost surely have J\'onsson terms.

Another question was the following:
If a random
finite algebra has a Maltsev term, will it have
(with probability 1) a majority term?
Letting $\MM$ denote the system of identities for a Maltsev term
(see subsection~\ref{ssec-perm}), our precise interpretation of the question
is the following:
Will a random
finite model of $\MM$ have a majority term with probability $1$?
We saw that a random finite model $\al{A}$ of $\MM$
has a $2$-element subalgebra with probability $1-e^{-2}$,
and is idemprimal with probability $e^{-2}$.
For $\al{A}$ to have a majority term, every 
$2$-element subalgebra of $\al{A}$ must have a majority term.
For this it is necessary that for every $2$-element subset
$B$ of $A$, $B$ is not a minority subalgebra of $\al{A}$, i.e.,
$B$ is not a subalgebra where
$f^{\al{B}}$ is the minority operation on $B$.
It is easy to check that for an $n$-element random model of $\MM$ and for
fixed $2$-element subset $B$ of $A$,
\[
\PrA\bigl(\text{$B$ is the universe of a minority subalgebra of $\al{A}$}\bigr)
=1/n^2.
\]
Therefore the same argument as in the proof of Theorem~\ref{thm-subalg}(2)
yields that
\[
\PrA\bigl(\text{$\al{A}$ has no $2$-element minority subalgebra}\bigr)
=\left(1-\frac{1}{n^2}\right)^{\binom{n}{2}},
\]
so the probability that a random finite model of $\MM$
has no $2$-element minority subalgebra is $e^{-1/2}$.
Hence, the probability that a random finite model of $\MM$
has a $2$-element minority subalgebra is $1-e^{-1/2}$.
This implies that a random finite model of $\MM$ will fail to have a majority
term with probability at least $1-e^{-1/2}\approx .39$.

\appendix
\section{Proof of Lemma~\ref{lm-murskii}}
\label{sec-appendix}

Let $\zeta_n(k)$ denote the $k$-th summand on the left hand side in
  \eqref{eq-murskii}, that is,
  \[
  \zeta_n(k):=\binom{n}{k}\left(\frac{k}{n}\right)^{\binom{k}{2}}.
  \]
Or goal is to prove that
\begin{equation}
  \label{eq-murskii-new}
\lim_{n\to\infty} \sum_{k=4}^{n-1}\zeta_n(k)=0.
\end{equation}

  \begin{clm}
  \label{clm-murskii1}
For arbitrary constants $0<u<v<1$ we have
$\displaystyle
\lim_{n\to\infty}\sum_{un\le k\le vn}\zeta_n(k)=0.
$
\end{clm}
    
\begin{proof}[Proof of Claim~\ref{clm-murskii1}]
  For $un\le k\le vn$,
  \[
  \zeta_n(k)
  =\binom{n}{k}\left(\frac{k}{n}\right)^{\binom{k}{2}}
  \le 2^n v^{un(un-1)/2}
  =\Bigl(2(\sqrt{v})^{u(un-1)}\Bigr)^n.
  \]
  Since $(\sqrt{v})^{u(un-1)}<\frac{1}{3}$ for large enough $n$, we get that
  \[
  \sum_{un\le k\le vn}\zeta_n(k)\le n\Bigl(\frac{2}{3}\Bigr)^n \to 0
  \qquad\text{as $n\to\infty$.}
  \]
  
\kern-15pt
\qedsymbdiamond  
\end{proof}

To establish \eqref{eq-murskii-new},
it remains to find $u,v$ with $0<u<v<1$ such that
\[
\sum_{4\le k< un}\zeta_n(k)\to 0
\quad\text{and}\quad
\sum_{vn< k< n}\zeta_n(k)\to 0
\quad\text{as $n\to\infty$.}
\]
This will be accomplished in Claims~\ref{clm-murskii2} and
\ref{clm-murskii3} below.

\begin{clm}
  \label{clm-murskii2}
  For $u=\frac{1}{2}$ we have that
$\displaystyle
  \lim_{n\to\infty}\sum_{4\le k< un}\zeta_n(k)= 0.
$
\end{clm}  

\begin{proof}[Proof of Claim~\ref{clm-murskii2}]
  The following estimates show that the sequence $\zeta_n(k)$ ($k=4,5,\dots$)
  is decreasing for $4\le k\le \frac{1}{2}n$:
  \begin{multline*}
    \frac{\zeta_n(k+1)}{\zeta_n(k)}
    =\frac{n-k}{k+1}\cdot\left(\frac{k+1}{n}\right)^k
    \cdot\left(\frac{k+1}{k}\right)^{\binom{k}{2}}\\
    =\frac{n-k}{n}\cdot\left(\frac{k+1}{n}\right)^{k-1}
    \cdot\left(\Bigl(1+\frac{1}{k}\Bigr)^k\right)^{\frac{k-1}{2}}
    \le\left(\frac{k+1}{n}\right)^{k-1}\cdot e^{\frac{k-1}{2}}\\
    <\left(\frac{k+1}{n}\right)^{k-1}\cdot 2^{k-1}\le 1
    \qquad\text{if}\quad k+1\le\frac{n}{2}.
  \end{multline*}
  The first term of the sequence is
  \[
  \zeta_n(4)=\binom{n}{4}\left(\frac{4}{n}\right)^6
  \le\frac{4^6}{4!}\cdot\frac{1}{n^2},
  \]
  hence
  \[
  \sum_{4\le k\le un}\zeta_n(k)\le un\cdot \frac{4^6}{4!}\cdot\frac{1}{n^2}
  \to 0
  \qquad\text{as $n\to\infty$.}
  \]

\kern-15pt  
\qedsymbdiamond  
\end{proof}

\begin{clm}
  \label{clm-murskii3}
  There exists a constant $v$ with $\frac{1}{2}<v<1$ such that
  \[
  \lim_{n\to\infty}\sum_{vn< k< n}\zeta_n(k)= 0.
  \]
\end{clm}  

\begin{proof}[Proof of Claim~\ref{clm-murskii3}]
  As in the proof of \cite[Lemma 6.22C]{bergman}, letting
  $\ell:=\lfloor vn\rfloor$ we have that
  \[
  \binom{n}{\ell}<\sqrt{n}\left(\Bigl(v-\frac{1}{n}\Bigr)^{v}
  \Bigl(1-v\Bigr)^{1-v}\Bigl(1-v\Bigr)^{\frac{1}{n}}\right)^{-n}.
  \]
  Now let $k$ be such that $(\frac{n}{2}<)\,vn<k<n$.
  We may assume without loss of generality that $n>8$, so $k\ge4$, and hence
  $\binom{k}{2}\ge\frac{k^2}{3}$.
  Thus,
  \[
  \left(\frac{k}{n}\right)^{\binom{k}{2}}
  \le\left(\frac{n-1}{n}\right)^{\binom{k}{2}}
  \le\left(\frac{n-1}{n}\right)^{\frac{k^2}{3}}
  \le\left(\frac{n-1}{n}\right)^{\frac{v^2n^2}{3}}
  =\left(\Bigl(1-\frac{1}{n}\Bigr)^{n}\right)^{\frac{v^2n}{3}}
  <\left(\frac{1}{e}\right)^{\frac{v^2n}{3}}.  
  \]
  Since $k\ge vn\ge\ell$ and $vn>\frac{n}{2}$, it follows that
  $\binom{n}{k}\le\binom{n}{\ell}$. Therefore, by combining the previous
  estimates we obtain that
  \begin{align*}
  \zeta_n(k)=\binom{n}{k}\left(\frac{k}{n}\right)^{\binom{k}{2}}
  & {}\le \sqrt{n}\left(\Bigl(v-\frac{1}{n}\Bigr)^{v}
  \Bigl(1-v\Bigr)^{1-v}\Bigl(1-v\Bigr)^{\frac{1}{n}}\right)^{-n}
  \left(\frac{1}{e}\right)^{\frac{v^2n}{3}}\\
  & {}= \sqrt{n}\left(
  \frac{1}{\Bigl(v-\frac{1}{n}\Bigr)^{v}
    \Bigl(1-v\Bigr)^{1-v}e^{\frac{v^2}{3}}}\right)^n(1-v)^{-1}\\
  & {}\le \sqrt{n}\left(
  \frac{1}{\Bigl(v-\frac{1}{16}\Bigr)^{v}
    \Bigl(1-v\Bigr)^{1-v}e^{\frac{v^2}{3}}}\right)^n(1-v)^{-1}\\
  & \qquad\qquad\qquad\qquad\qquad\qquad\qquad\qquad\qquad
  \text{if $n\ge16$.}
  \end{align*}
  Let $w:=\Bigl(v-\frac{1}{16}\Bigr)^{v}\Bigl(1-v\Bigr)^{1-v}e^{\frac{v^2}{3}}$.
  It can be checked that there exists $v$ with $\frac{1}{2}<v<1$ such that
  $w>1$; for example, $v=.95$ works.
  Thus,
  \[
  \sum_{vn<k<n}\zeta_n(k)
  \le\sum_{vn<k<n}\frac{\sqrt{n}}{1-v}\left(\frac{1}{w}\right)^n\to0
  \qquad\text{as $n\to\infty$.}
  \]
  
\kern-15pt  
\qedsymbdiamond  
\end{proof}

This completes the proof of Lemma~\ref{lm-murskii}.

\bibliographystyle{plain}

\end{document}